\documentclass[11pt,leqno]{amsart}
\usepackage{amsmath}
\usepackage{amssymb}
\usepackage{amsthm}
\usepackage[dvips]{graphics}
\usepackage{graphicx,color}
\usepackage{blindtext}
\usepackage{url}
\usepackage{hyperref}
\hypersetup{
    colorlinks=true,
    linkcolor=blue,
    filecolor=magenta,      
    urlcolor=cyan,
}

\renewcommand{\arraystretch}{1.2}
\newtheorem{theorem}{Theorem}[section]
\newtheorem{lemma}{Lemma}[section]
\newtheorem{corollary}{Corollary}[section]

\newtheorem{proposition}{Proposition}[section]
\newtheorem{definition}{Definition}[section]

\newtheorem{example}{Example}[section]

\newcommand{\ignore}[1]{}

\renewcommand{\arraystretch}{1.2}
\newtheorem{alg}{Algorithm}

\def\ba{\begin{array}}
\def\ea{\end{array}}
\def\beq{\begin{equation}}
\def\eeq{\end{equation}}
\def\bea{\begin{eqnarray}}
\def\eea{\end{eqnarray}}
\def\beann{\begin{eqnarray*}}
\def\eeann{\end{eqnarray*}}

\def\R{\mathbb{R}}

\def\Diag{\textup{Diag}}

\def\inte{\textup{int}}

\def\rank{\textup{rank}}
\def\trace{\textup{Tr}}
\def\Null{\textup{null}}

\def\ln{\textup{ln}}
\def\exp{\textup{exp}}

\newcommand{\symm}{\mathbb S}

\newcommand{\tx}[1]{\texttt{#1}}

\title{\bf Domain-Driven Solver (DDS) Version 2.0:\\ a MATLAB-based Software Package for Convex Optimization Problems in Domain-Driven Form}
\author{
Mehdi Karimi \and Levent Tun\c{c}el}

\thanks{
Mehdi Karimi (m7karimi@uwaterloo.ca) and Levent Tun\c{c}el (ltuncel@math.uwaterloo.ca): Department of
Combinatorics and Optimization, University
of Waterloo, Waterloo, Ontario N2L 3G1, Canada. \\ Research of the authors was supported in part by Discovery Grants from NSERC and by U.S. Office of Naval Research under award numbers: N00014-12-1-0049, N00014-15-1-2171 and N00014-18-1-2078.}

\date{August 7, 2019 (arXiv: 1908.03075); revised: \today}

\let\oldtocsection=\tocsection
 
\let\oldtocsubsection=\tocsubsection
 
\let\oldtocsubsubsection=\tocsubsubsection
 
\renewcommand{\tocsection}[2]{\hspace{0em}\oldtocsection{#1}{#2}}
\renewcommand{\tocsubsection}[2]{\hspace{2em}\oldtocsubsection{#1}{#2}}
\renewcommand{\tocsubsubsection}[2]{\hspace{3em}\oldtocsubsubsection{#1}{#2}}

\begin{document}
{\begin{abstract}
Domain-Driven Solver (DDS) is a MATLAB-based software package for convex optimization problems in Domain-Driven form \cite{karimi_arxiv}. The current version of DDS accepts every combination of the following function/set constraints: (1)  symmetric cones (LP, SOCP, and SDP); (2) quadratic constraints that are SOCP representable; (3) direct sums of an
arbitrary collection of 2-dimensional convex sets defined as the epigraphs of univariate convex
functions (including as special cases geometric programming and entropy programming); (4) generalized power cone; (5) epigraphs of matrix norms (including
as a special case minimization of nuclear norm over a linear subspace); (6) vector relative entropy; (7) epigraphs of quantum
entropy and quantum relative entropy; and (8) constraints involving hyperbolic polynomials. DDS is a practical implementation of the infeasible-start primal-dual algorithm designed and analyzed in \cite{karimi_arxiv}. This manuscript contains the users' guide, as well as theoretical results needed for the implementation of the algorithms. To help the users, we included many examples. We also discussed some implementation details and techniques we used to improve the efficiency and further expansion of the software to cover the emerging classes of convex optimization problems.  
\end{abstract} }
\maketitle

\pagestyle{myheadings} \thispagestyle{plain}
\markboth{KARIMI and TUN{\c C}EL}
{Domain-Driven Solver (DDS)}
\setcounter{tocdepth}{2}
\newpage
{
\tableofcontents
\newpage
}
\section{Introduction}\label{sec:intro}
The code DDS (Domain-Driven Solver) solves convex optimization problems of the form
\begin{eqnarray} \label{main-p}
\inf _{x} \{\langle c,x \rangle :  Ax \in D\},
\end{eqnarray}
where $x \mapsto Ax : \R^n \rightarrow \R^m$ is a linear embedding, $A$ and $c \in \mathbb \R^n$ are given, and $D \subset \mathbb \R^m$ is a closed convex set 
defined as the closure of the domain of a $\vartheta$-\emph{self-concordant (s.c.) barrier} $\Phi$ \cite{interior-book,lectures-book}. In practice, the set $D$ is typically formulated as $D=D_1 \oplus \cdots \oplus D_\ell$, where $D_i$ is associated with a s.c.\ barrier $\Phi_i$, for $i \in \{1,\ldots,\ell\}$. Every input constraint for DDS may be thought of as either the convex set it defines or the corresponding s.c.\ barrier. 

The current version of DDS accepts many functions and set constraints as we explain in this article. If a user has a nonlinear convex objective function $f(x)$ to minimize, one can introduce a new variable $x_{n+1}$ and minimize a linear function $x_{n+1}$ subject to the convex set constraint $f(x) \leq x_{n+1}$ (and other convex constraints in the original optimization problem). As a result, in this article we will talk about representing functions and convex set constraints interchangeably.  The algorithm underlying the code also uses the Legendre-Fenchel (LF) conjugate $\Phi_*$ of $\Phi$ if it is computationally efficient to evaluate $\Phi_*$ and its first (and hopefully second) derivative. 
Any new discovery of a s.c.\ barrier allows DDS to expand the classes of convex optimization problems it can solve as any new s.c.\ barrier $\Phi$ with a computable LF conjugate can be easily added to the  code. DDS is a practical implementation of the primal-dual algorithm designed and analyzed in \cite{karimi_arxiv}, which has the current best iteration complexity bound available for conic formulations. Stopping criteria for DDS and the way DDS suggests the status (``has an approximately optimal solution", ``is infeasible", ``is unbounded", etc.) is based on the analyses in \cite{karimi_status_arxiv}. 

Even though there are similarities between DDS and some modeling systems such as CVX \cite{cvx} (such as the variety input constraints), there are major differences, including:
\begin{itemize}
\item DDS is not just a modeling system and it uses its own algorithm.
The algorithm used in DDS is an infeasible-start primal-dual path-following algorithm, and is of predictor corrector type \cite{karimi_arxiv}. 
\item The modeling systems for convex optimization that rely on SDP solvers have to use approximation for set constraints which are not efficiently representable by spectrahedra (for example epigraph of functions involving $exp$ or $\ln$ functions). However, DDS uses a s.c.\ barrier specifically well-suited to each set constraint without such approximations.  This enables DDS to return proper certificates for all the input problems.
\item As far as we know, some set constraints such as hyperbolic ones, are not accepted by other modeling systems. Some other constraints such as epigraphs of matrix norms, and those involving quantum entropy and quantum relative entropy are handled more efficiently than other existing options.
\end{itemize}

The main part of the article is organized to be more as a users' guide, and contains the minimum theoretical content required by the users. In Section \ref{sec:how}, we explain how to install and use DDS. Sections \ref{sec:LP}-\ref{sec:eq} explain how to input different types of set/function constraints into DDS. Section \ref{sec:num} contains several numerical results and tables. Many theoretical foundations needed by DDS are included in the appendices. The LHS matrix of the linear systems of equations determining  the predictor and corrector steps have a similar form. In Appendix  \ref{appen:LS}, we explain how such linear systems are being solved for DDS. Some of the constraints accepted by DDS are involving matrix functions and many techniques are used in DDS to efficiently evaluate these functions and their derivatives, see Appendices \ref{sec:imp} and \ref{appen:quantum}. DDS heavily relies on efficient evaluation of LF conjugates of s.c.\ barrier functions. For some of the functions, we show in the main text how to efficiently calculate the LF conjugate. For some others, the calculations are shifted to the appendices, for example Appendix \ref{app:LF-univar}. DDS uses interior-point methods, and gradient and Hessian of the s.c.\ barriers are building blocks of the underlying linear systems. Explicit formulas for the gradient and Hessian of some s.c.\ barriers and their LF conjugates are given in Appendix \ref{formulas}. Appendix \ref{sec:other-solvers} contains a brief comparison of DDS with some popular solvers and modeling systems. 

\subsection{Installation} 
The current version of DDS is written in MATLAB. This version is available from the websites:

\noindent \url{http://www.math.uwaterloo.ca/~m7karimi/DDS.html} \\
\url{https://github.com/mehdi-karimi-math/DDS}

To use DDS, the user can follow these steps:
\begin{itemize}
\item unzip DDS.zip;
\item run MATLAB in the directory DDS;
\item run the m-file \tx{DDS\_startup.m}.
\item (optional) run the prepared small examples \tx{DDS\_example\_1.m} and \tx{DDS\_example\_2.m}.
\end{itemize}
The prepared examples contain many set constraints accepted by DDS and running them without error indicates that DDS is ready to use. There is a directory \tx{Text\_Examples} in the DDS package which includes many examples on different classes of convex optimization problems. 
 \section{How to use the DDS code} \label{sec:how}
 In this section, we explain the format of the input for many popular classes of optimization problems. In practice, we typically have $D=\bar D -b$, where $\inte \bar D$ is the domain of a ``canonical" s.c.\ barrier and $b \in \R^m$. For example, for LP, we typically have $D=\R^n_++b$, where $b \in \R^m$ is given as part of the input data, and $-\sum_{i=1}^{n}\ln(x_i)$ is a s.c.\ barrier for $\R^n_+$. 
 The command in MATLAB that calls DDS is 
\begin{verbatim}
[x,y,info]=DDS(c,A,b,cons,OPTIONS);
\end{verbatim} 

\noindent {\bf Input Arguments:} \\
\noindent \tx{cons}:  A cell array that contains the information about the type of constraints. \\
\noindent \tx{c,A,b}: Input data for DDS: $A$ is the coefficient matrix, $c$ is the objective vector, $b$ is the RHS vector (i.e., the shift in the definition of the convex domain $D$). \\
\noindent \tx{OPTIONS} (optional): An array which contains information about the tolerance  and initial points.  

\noindent {\bf Output Arguments:} \\
\noindent \tx{x}: Primal point. \\
\noindent \tx{y}: Dual point which is a cell array. Each cell contains the dual solution for the constraints in the corresponding cell  in \tx{A}. \\
\noindent \tx{info}: A structure array containing performance information such as \tx{info.time}, which returns the CPU time for solving the problem. 

Note that in the Domain-Driven setup, the primal problem is the main problem, and the dual problem is implicit for the user. This implicit dual problem is: 
\begin{eqnarray} \label{main-d}
\inf _{y} \{\delta_*(y|D) :  A^\top y = -c, y \in D_*\},
\end{eqnarray}
where $\delta_*(y|D) := \sup\{ \langle y,z \rangle :  z \in D\},$ is the \emph{support function} of $D$, and $D_*$ is defined as 
\begin{eqnarray} \label{eq:leg-conj-2} 
D_*:=\{y:  \langle y , h \rangle  \leq 0, \ \ \forall h \in \text{rec}(D)\},
\end{eqnarray}
where $\text{rec}(D)$ is the \emph{recession cone} of $D$. For the convenience of users, there is a built-in function in DDS package to calculate the dual objective value of the returned $y$ vector:
\begin{verbatim}
z=dual_obj_value(y,b,cons);
\end{verbatim}
For a primal feasible point $x \in \R^n$ which satisfies $Ax \in D$ and a dual feasible point $y \in D_*$, the duality gap is defined in \cite{karimi_arxiv} as
\begin{eqnarray} \label{eq:duality-gap}
\langle c,x \rangle + \delta_*(y|D).
\end{eqnarray}
It is proved in \cite{karimi_arxiv} that the duality gap is well-defined and zero duality gap implies optimality. If DDS returns status ``solved" for a problem (\tx{info.status=1}), it means \tx{(x,y)} is a pair of approximately feasible primal and dual points, with duality gap close to zero (based on tolerance). If \tx{info.status=2}, the problem is suspected to be unbounded and the returned \tx{x} is a point, approximately primal feasible with very small objective value ($\langle c,x \rangle \leq -1/tol$). If \tx{info.status=3}, problem is suspected to be infeasible, and the returned \tx{y}  in $D_*$ approximately satisfies $A^\top y=0$ with $\delta_*(y|D) <0$. If  \tx{info.status=4}, problem is suspected to be ill-conditioned. 

The user is not required to input any part of the \tx{OPTIONS} array. The default settings are:
\begin{itemize}
\item $tol=10^{-8}$.
\item The initial points $x^0$ and $z^0$ for the infeasible-start algorithm are chosen such that, assuming $D=D_1 \oplus \cdots \oplus D_\ell$, the $i$th part of  $Ax^0+z^0$ is a canonical point in $\inte D_i$.   
\end{itemize}
However, if a user chooses to provide \tx{OPTIONS} as an input, here is how to define the desired parts: 
 \tx{OPTIONS.tol} may be given as the desired tolerance, otherwise the default $tol:=10^{-8}$ is used. \tx{OPTIONS.x0} and \tx{OPTIONS.z0} may be defined as the initial points as any pair of  points $x^0 \in \R^n$ and $z^0 \in \R^m$ that satisfy $Ax^0+z^0 \in \inte D$. If only \tx{OPTIONS.x0} is given, then $x^0$ must satisfy $Ax^0 \in \inte D$. In other words, \tx{OPTIONS.x0} is a point that strictly satisfies all the constraints.

In the following sections, we discuss the format of each input function/set constraint.
Table \ref{table:order} shows the classes of function/set constraints the current version of DDS accepts, plus the abbreviation we use to represent the constraint. 
From now on, we assume that the objective function is ``$\inf \ \langle c,x\rangle$", and we show how to add various function/set constraints. Note that \tx{A}, \tx{b}, and \tx{cons} are cell arrays in MATLAB. \tx{cons(k,1)} represents type of the $k$th block of constraints by using the abbreviations of Table \ref{table:order}. For example, \tx{cons(2,1)='LP'} means that the second block of constraints are linear inequalities. It is advisable to group the constraints of the same type in one block, but not necessary.

 \begin{center}
 \begin{table} [h]
 \caption{Function/set constraints the current version of DDS accepts, and their abbreviations. }
 \label{table:order}
 \begin{tabular} {|c|c|}
 \hline
 {\bf function/set constraint} &  {\bf abbreviation} \\ \hline
  LP   &  LP\\ \hline
 SOCP  & SOCP\\ \hline
 Rotated SOCP  & SOCPR\\ \hline
 SDP  & SDP \\ \hline
 Generalized Power Cone & GPC \\ \hline
 Quadratic Constraints & QC \\ \hline
Epigraph of a Matrix Norm  &  MN \\ \hline
 Direct sum of 2-dimensional sets  &  \\
 (geometric, entropy, and $p$-norm & TD \\
 programming)   &  \\ \hline 
 Quantum Entropy  & QE \\ \hline
 Quantum Relative Entropy  & QRE \\ \hline
 Relative Entropy & RE \\ \hline
 Hyperbolic Polynomials & HB \\ \hline
 Equality Constraints & EQ \\ \hline
 \end{tabular}
 \end{table}
 \end{center}
 
 \section{LP, SOCP, SDP} \label{sec:LP}
 
\subsection{Linear programming (LP) and second-order cone programming (SOCP)}
Suppose we want to add $\ell$ LP constraints of the form
\begin{eqnarray}  \label{LP-1-1}
 A_L^i x + b_L^i \geq 0, \ \ \ i\in\{1,\ldots,\ell\}, 
\end{eqnarray}
where $A_L^i$ is an $m_L^i$-by-$n$ matrix, as the $k$th block of constraints. Then, we define
\begin{eqnarray}\label{LP-2-1}
\texttt{A\{k,1\}}=\left [ \begin{array} {c} A_L^1 \\ \vdots \\ A_L^\ell\end{array}\right], \ \ \ \tx{b\{k,1\}}=\left [ \begin{array} {c} b_L^1 \\ \vdots\\  b_L^\ell\end{array}\right] \nonumber \\
 \texttt{cons\{k,1\}='LP'},  \ \ \ \texttt{cons\{k,2\}}=[m_L^1, \ldots , m_L^\ell].
\end{eqnarray}
Similarly to add $\ell$ SOCP constraints of the form
\begin{eqnarray}  \label{LP-1}
\|A_S^i x + b_S^i\| \leq (g_S^i)^\top x + d_S^i, \ \ \ i\in\{1,\ldots,\ell\},
\end{eqnarray}
where $A_S^i$ is an $m_S^i$-by-$n$ matrix for $i=\in\{1,\ldots,\ell\}$,  as the $k$th block, we define
\begin{eqnarray}\label{LP-2}
\texttt{A\{k,1\}}=\left [ \begin{array} {c} (g_S^1)^\top \\ A_S^1 \\ \vdots \\  (g_S^\ell)^\top \\ A_S^\ell\end{array}\right], \ \ \ \tx{b\{k,1\}}=\left [ \begin{array} {c} d_S^1 \\ b_S^1 \\ \vdots \\  d_S^\ell \\ b_S^\ell\end{array}\right] \nonumber \\
 \texttt{cons\{k,1\}='SOCP'},  \ \ \ \texttt{cons\{k,2\}}=[m_S^1, \ldots	 , m_S^\ell].
\end{eqnarray}
Let us see an example:
\begin{example} Suppose we are given the problem:
\begin{eqnarray} \label{LP-3}
&\min& c^\top x \nonumber \\
& \text{s.t.}&  [-2,1] x \leq 1, \nonumber  \\
&&  \left\|\left [\begin{array} {cc}  2 & 1 \\ 1 & 3 \end{array} \right ] x + \left [\begin{array} {c} 3 \\ 4 \end{array} \right ] \right\|_2 \leq 2.
\end{eqnarray}
Then we define
\begin{eqnarray*} \label{LP-4}
\tx{cons\{1,1\}='LP'}, \ \ \tx{cons\{1,2\}=[1]}, \ \ \tx{A\{1,1\}}=\left [ \begin{array} {cc} 2 & -1 \end{array}\right], \ \ \ \tx{b\{1,1\}}=\left [ \begin{array} {c} 1 \end{array}\right], \nonumber \\
\tx{cons\{2,1\}='SOCP'}, \ \ \tx{cons\{2,2\}=[2]}, \ \  \tx{A\{2,1\}}=\left [ \begin{array} {cc}  0 & 0 \\ 2 & 1 \\ 1 & 3 \end{array}\right], \ \ \ \tx{b\{2,2\}}=\left [ \begin{array} {c}  2 \\ 3 \\ 4 \end{array}\right].
\end{eqnarray*}
\end{example}
The s.c.\ barriers and their LF conjugates being used in DDS for these constraints are
\begin{eqnarray} \label{eqn:app-1}
&& \Phi(z)=-\ln(z), \ \ z \in \mathbb R_+, \ \ \ \Phi_*(\eta)=-1-\ln(-\eta),  \nonumber \\
&& \Phi(t,z)=-\ln(t^2-z^\top z), \ \ \ \Phi_*(\eta,w)=-2+\ln(4)-\ln(\eta^2-w^\top w). 
\end{eqnarray}

DDS also accepts constraints defined by the rotated second order cones:
\begin{eqnarray}
\{(z,t,s) \in \R^n \oplus \R \oplus \R: \|z\|^2 \leq ts, \ t \geq 0, \ s \geq 0\},
\end{eqnarray}
which is handled by the s.c.\ barrier $-\ln(ts-z^\top z)$. The abbreviation we use is \tx{'SOCPR'}.
To add $\ell$ rotated SOCP constraints of the form
\begin{eqnarray}  \label{LP-1-SOCPR}
&&\|A_S^i x + b_S^i\|_2^2 \leq ((g_S^i)^\top x + d_S^i)((\bar g_S^i)^\top x + \bar d_S^i), \ \ \ i\in\{1,\ldots,\ell\},  \nonumber \\
&& (g_S^i)^\top x + d_S^i\geq 0, \ \ (\bar g_S^i)^\top x + \bar d_S^i \geq 0, 
\end{eqnarray}
where $A_S^i$ is an $m_S^i$-by-$n$ matrix for $i \in\{1,\ldots,\ell\}$,  as the $k$th block, we define
\begin{eqnarray}\label{LP-2-SOCPR}
\texttt{A\{k,1\}}=\left [ \begin{array} {c} (g_S^1)^\top \\ (\bar g_S^1)^\top \\ A_S^1 \\ \vdots \\  (g_S^\ell)^\top \\ (\bar g_S^\ell)^\top  \\ A_S^\ell\end{array}\right], \ \ \ \tx{b\{k,1\}}=\left [ \begin{array} {c} d_S^1 \\ \bar d_S^1 \\ b_S^1 \\ \vdots \\  d_S^\ell \\ \bar d_S^\ell \\ b_S^\ell\end{array}\right] \nonumber \\
 \texttt{cons\{k,1\}='SOCPR'},  \ \ \ \texttt{cons\{k,2\}}=[m_S^1, \ldots , m_S^\ell].
\end{eqnarray}

\subsection{Semidefinite programming (SDP)}
Consider $\ell$ SDP constraints in standard inequality (linear matrix inequality (LMI)) form:
\begin{eqnarray} \label{SDP-1}
  F^i_0+x_1 F^i_1+ \cdots+x_n F^i_n \succeq 0, \ \ \ i\in\{1,\ldots,\ell\}.
\end{eqnarray}
$F^i_j$'s are $n_i$-by-$n_i$ symmetric matrices. The above optimization problem is in the matrix form. To formulate it in our setup, we need to write it in the vector form. DDS has two internal functions \tx{sm2vec} and \tx{vec2sm}. \tx{sm2vec} takes an $n$-by-$n$ symmetric matrix and changes it into a vector in $\R^{n^2}$  by stacking the columns of it on top of one another in order. 
\tx{vec2sm} changes a vector into a symmetric matrix such that
\begin{eqnarray}  \label{SDP-2}
\tx{vec2sm(sm2vec(X))=X}.
\end{eqnarray}
By this definition, it is easy to check that for any pair of $n$-by-$n$ symmetric matrices $X$ and $Y$ we have
\begin{eqnarray} \label{SDP-3}
\langle X,Y \rangle = \tx{sm2vec(X)}^ \top \tx{sm2vec(Y)}.
\end{eqnarray}
To give \eqref{SDP-1} to DDS as the $k$th input block, we define:
\begin{eqnarray} \label{SDP-4}
&&\tx{A\{k,1\}}:=\left [\begin{array}{c} \tx{sm2vec}(F^1_1), \cdots, \tx{sm2vec}(F^1_n)  \\ \vdots \\ \tx{sm2vec}(F^\ell_1), \cdots, \tx{sm2vec}(F^\ell_n)\end{array} \right ], \ \ \ b\{k,1\}:=\left [ \begin{array}{c}\tx{sm2vec}(F^1_0)\\ \vdots \\ \tx{sm2vec}(F^\ell_0) \end{array} \right], \nonumber \\
&& \tx{cons\{k,1\}='SDP'}   \ \ \ \tx{cons\{k,2\}}=[n^1,\ \ldots \ , n^\ell].
\end{eqnarray}
The s.c.\ barrier used in DDS for SDP is the well-known function $-\ln(\det(X))$ defined on the convex cone of symmetric positive definite matrices. 

\begin{example} Assume that we want to find scalars $x_1$, $x_2$, and $x_3$ such that $x_1+x_2+x_3 \geq 1$ and  the maximum eigenvalue of $A_0+x_1A_1+x_2A_2+x_3A_3$ is minimized, where
\begin{eqnarray} \nonumber 
A_0=\left [ \begin{array}{ccc}2& -0.5& -0.6 \\ -0.5 & 2 & 0.4 \\-0.6 & 0.4 & 3 \end{array} \right], \ 
A_1=\left [ \begin{array}{ccc}0& 1& 0 \\ 1 & 0 & 0 \\0 & 0 & 0 \end{array} \right], \ 
A_2=\left [ \begin{array}{ccc}0& 0& 1 \\ 0 & 0 & 0 \\1 & 0 & 0 \end{array} \right], \
A_3=\left [ \begin{array}{ccc}0& 0& 0 \\ 0 & 0 & 1 \\0 & 1 & 0 \end{array} \right].
\end{eqnarray}
We can write this problem as 
\begin{eqnarray}
&\min& t \nonumber \\
&s.t.& -1+x_1+x_2+x_3 \geq 0, \nonumber \\
&& tI-(A_0+x_1A_1+x_2A_2+x_3A_3) \succeq 0.
\end{eqnarray}
To solve this problem, we define:
\begin{eqnarray*}
\tx{cons\{1,1\}='LP'}, \ \ \tx{cons\{1,2\}}=[1], \ \ \tx{cons\{2,1\}='SDP'}, \ \ \tx{cons\{2,2\}}=[3], \\
\tx{A\{1,1\}}=\left [\begin{array}{cccc} 1&1&1&0  \end{array} \right], \ \ \ \tx{b\{1,1\}}=\left [\begin{array}{c} -1 \end{array} \right], \\
\tx{A\{2,1\}}=\left [\begin{array}{cccc}  0&0&0&1 \\ -1&0&0&0 \\0&-1&0&0 \\ -1&0&0&0 \\ 0&0&0&1 \\0&0&-1&0 \\ 0 & -1& 0&0 \\ 0&0&-1&0 \\ 0&0&0&1 \end{array} \right], \ \ \ \tx{b\{2,1\}}=\left [\begin{array}{c}  -2 \\ 0.5  \\0.6 \\ 0.5 \\ -2 \\-0.4 \\ 0.6 \\ -0.4 \\ -3 \end{array} \right],
   \\
\tx{c}=(0,0,0,1)^\top. 
\end{eqnarray*}
Then \tx{DDS(c,A,b,cons)} gives the answer $x=(1.1265,0.6,-0.4,3)^\top$, which means the minimum largest eigenvalue is $3$. 
\end{example}
\section{Quadratic constraints}
Suppose we want to add the following constraints to DDS:
\begin{eqnarray} \label{QC-1}
x^\top A_i^\top Q_i A_i x + b_i^\top x + d_i  \leq 0, \ \ \ i \in \{1,\ldots,\ell\},
\end{eqnarray}
where each $A_i$ is $m_i$-by-$n$ with rank $n$, and $Q_i \in \mathbb S^{m_i}$. In general, this type of constraints may be non-convex and difficult to handle. Currently, DDS handles two cases:
\begin{itemize}
\item $Q_i$ is positive semidefinite,
\item $Q_i$ has exactly one negative eigenvalue. 
In this case, DDS considers the intersection of the set of points satisfying \eqref{QC-1} and a shifted \emph{hyperbolicity cone} defined by the quadratic inequality $y^\top Q_i y \leq 0$. 
 
\end{itemize}
To give constraints in \eqref{QC-1} as input to DDS as the $k$th block, we define
\begin{eqnarray}
\tx{A\{k,1\}}=\left [ \begin{array} {c} b_1^\top \\ A_1 \\ \vdots \\  b_l^\top \\ A_\ell\end{array}\right], \ \ \ \tx{b\{k,1\}}=\left [ \begin{array} {c} d_1 \\ 0 \\ \vdots \\  d_\ell \\ 0 \end{array}\right] \nonumber \\
\tx{cons\{k,1\}='QC'} \ \ \tx{cons\{k,2\}}=[m_1,\ldots,m_\ell], \nonumber \\
\tx{cons\{k,3,i\}}=Q_i, \ \ i \in \{1,\ldots,\ell\}. 
\end{eqnarray}
If \tx{cons\{k,3\}} is not given as the input, DDS takes all $Q_i$'s to be identity matrices. 

If $Q_i$ is positive semidefinite, then the corresponding constraint in \eqref{QC-1} can be written as
\begin{eqnarray}
&& u^\top u + w + d \leq 0 \nonumber \\
&& u := R_i A_i x,  \ \ \ w:=b_i^\top x, \ \ d:= d_i,
\end{eqnarray}
where $Q_i=R_i^\top R_i$ is a Cholesky factorization of $Q_i$. 
We associate the following s.c.\ barrier and its LF conjugate to such quadratic constraints:
\begin{eqnarray}
\Phi(u,w) &=& - \ln (- (u^\top u + w + d)), \nonumber \\
\Phi_*(y,\eta) &=& \frac{y^\top y}{4\eta}-1-d\eta-\ln (\eta).
\end{eqnarray}
If $Q_i$ has exactly one negative eigenvalue with eigenvector $v$, then $-y^\top Q_i y$ is a hyperbolic polynomial with respect to $v$. The hyperbolicity cone is the connected component of $y^\top Q_i y \leq 0$ which contains $v$ and $-\ln(-y^\top Q_i y)$ is a s.c.\ barrier for this cone. 

If for any of the inequalities in \eqref{QC-1}, $Q_i$ has exactly one negative eigenvalue while $b_i=0$ and $d_i=0$, DDS considers the hyperbolicity cone defined by the inequality as the set constraint.

\section{Generalized Power Cone}
We define the $(m,n)$-generalized power cone with parameter $\alpha$ as
\begin{eqnarray} \label{eq:GPC-1}
K^{(m,n)}_\alpha := \left\{(s,u) \in \R_+^m \oplus \R^n: \prod_{i=1}^m s_i^{\alpha_i}  \geq \|u\|_2 \right\},
\end{eqnarray} 
where $\alpha$ belongs to the simplex $\{\alpha \in \R^m_+: \sum_{i=1}^m \alpha_i = 1\}$. Note that the rotated second order cone is a special case where $m=2$ and $\alpha_1 = \alpha_2 = \frac 12$. Different s.c.\ barriers for this cone or special cases of it have been considered \cite{chares, tunccel2010self}. Chares conjectured a s.c.\ barrier for the cone in \eqref{eq:GPC-1}, which is proved in \cite{roy2018self}. This function that we use in DDS is:
\begin{eqnarray} \label{eq:GPC-2}
\Phi(s,u) = -\ln\left(\prod_{i=1}^m s_i^{2\alpha_i} - u^\top u \right) - \sum_{i=1}^{m} (1-\alpha_i) \ln(s_i).
\end{eqnarray}
To add generalized power cone constraints to DDS, we use the abbreviation 'GPC'. Therefore, if the $k$th block of constraints is GNC, we define \tx{cons\{k,1\}='GPC'}. Assume that we want to input the following $\ell$ constraints to DDS:
\begin{eqnarray} \label{eq:GPC-3}
(A_s^i x + b_s^i, A_u^i x+b_u^i)  \in K^{(m_i,n_i)}_{\alpha^i}, \ \ \ i \in  \{1,\ldots,\ell\},
\end{eqnarray}
where $A_s^i$, $b_s^i$, $A_u^i$, and $b_u^i$ are matrices and vectors of proper size. Then, to input these constraints as the $k$th block, we define \tx{cons\{k,2\}} as a MATLAB cell array of size $\ell$-by-$2$, each row represents one constraint. We then define:
\begin{eqnarray} \label{eq:GPC-4}
\tx{cons\{k,2\}\{i,1\}} &=& [m_i \ \ n_i],   \nonumber \\
\tx{cons\{k,2\}\{i,2\}} &=&  \alpha^i, \ \ \ \ i \in  \{1,\ldots,\ell\}.
\end{eqnarray}
For matrices $A$ and $b$, we define:
\begin{eqnarray} \label{eq:GPC-5}
\texttt{A\{k,1\}}=\left [ \begin{array} {c}  A_s^1 \\ A_u^1 \\ \vdots \\  A_s^\ell \\ A_u^\ell \end{array}\right], \ \ \ \tx{b\{k,1\}}=\left [ \begin{array} {c} b_s^1 \\ b_u^1 \\ \vdots \\ b_s^\ell \\ b_u^\ell \end{array}\right].
\end{eqnarray}

\begin{example}
Consider the following optimization problem with 'LP' and 'GPC' constraints:
\begin{eqnarray} \label{eq:GPC-6}
&\min&  -x_1-x_2-x_3 \nonumber \\
&s.t.&       \|x\| \leq (x_1+3)^{0.3} (x_2+1)^{0.3} (x_3+2)^{0.4}, \nonumber \\
&&           x_1,x_2,x_3 \geq 3.
\end{eqnarray}
Then we define:
\begin{eqnarray*} \label{eq:GPC-7}
&& \tx{cons\{1,1\}='GPC'}, \ \ \tx{cons\{1,2\}}=\{[3, \ \ 3], \ \   [0.3; 0.3; 0.4] \}  \\
&& \tx{A\{1,1\}}=\left [\begin{array} {cc} \tx{eye(3)} \\ \tx{eye(3)} \end{array}  \right], \ \ \tx{b\{1,1\}}= [3;1;2;0;0;0] \\
&& \tx{cons\{2,1\}='LP'}, \ \ \tx{cons\{2,2\}}=[3] \\
&& \tx{A\{2,1\}}=\left [ -\tx{eye(3)}  \right], \ \ \tx{b\{2,1\}}= [3;3;3] \\
&& c=[-1,-1,-1]. 
\end{eqnarray*}
\end{example}
\section{Epigraphs of matrix norms}
Assume that we have constraints of the form 
\begin{eqnarray} \label{EO2N-1}
&& X-UU^\top \succeq 0,  \nonumber \\
&& X=A_0+\sum_{i=1}^\ell x_i A_i, \nonumber \\
&& U=B_0+\sum_{i=1}^\ell x_i B_i,
\end{eqnarray}
where $A_i$, $i \in  \{1,\ldots,\ell\}$, are $m$-by-$m$ symmetric matrices, and $B_i$, $i \in  \{1,\ldots,\ell\}$, are $m$-by-$n$ matrices. 
The set $\{(Z,U) \in \mathbb S^m \oplus \R^{m \times n}: Z-UU^\top \succeq 0 \}$ is handled by the following s.c.\ barrier:
\begin{eqnarray} \label{EO2N-3}
\Phi(Z,U):=-\ln(\det(Z-UU^\top)),
\end{eqnarray}
 with LF conjugate
 \begin{eqnarray} \label{EO2N-4}
\Phi_*(Y,V)=-m-\frac{1}{4} \trace(V^\top Y^{-1} V)-\ln(\det(-Y)),
\end{eqnarray} 
where $Y \in \R^{m\times m}$ and $V \in \R^{m\times n}$ \cite{cone-free}. This constraint can be reformulated as an SDP constraint using a Schur complement. However, $\Phi(Z,U)$ is a $m$-s.c.\ barrier while the size of SDP reformulation is $m+n$. For the cases that $m \ll n$, using the Domain-Driven form may be advantageous. 
A special but very important application is minimizing the \emph{nuclear norm} of a matrix, which we describe in a separate subsection in the following.

DDS has two internal functions \tx{m2vec} and \tx{vec2m} for converting matrices (not necessarily symmetric) to vectors and vice versa. For an $m$-by-$n$ matrix $Z$, \tx{m2vec(Z,n)} change the matrix into a vector. \tx{vec2m(v,m)} reshapes a vector  $v$ of proper size to a matrix with $m$ rows. The abbreviation we use for epigraph of a matrix norm is MN. If the $k$th input block is of this type,  \tx{cons\{k,2\}} is a $\ell$-by-$2$ matrix, where $\ell$ is the number of constraints of this type, and each row is of the form $[m \ \ n]$. For each constraint of the form \eqref{EO2N-1}, the corresponding parts in $A$ and $b$ are defined as
\begin{eqnarray} \label{EO2N-5}
\tx{A\{k,1\}}=\left [\begin{array} {ccc} \tx{m2vec}(B_1,n) & \cdots &  \tx{m2vec}(B_\ell,n) \\  \tx{sm2vec}(A_1) & \cdots &  \tx{sm2vec}(A_\ell)  \end{array}  \right], \ \ \tx{b\{k,1\}}=\left [\begin{array} {c} \tx{m2vec}(B_0,n) \\  \tx{sm2vec}(A_0)  \end{array}  \right].
\end{eqnarray}
For implementation details involving epigraph of matrix norms, see Appendix \ref{sec:imp:mn}. Here is an example:
\begin{example}
Assume that we have matrices 
\begin{eqnarray} \label{EO2N-6}
U_0=\left [\begin{array} {ccc} 1 & 0 & 0 \\0 & 1 & 1  \end{array}  \right], \ \ U_1=\left [\begin{array} {ccc} -1 & -1 & 1 \\0 & 0 & 1  \end{array}  \right], \ \ U_2=\left [\begin{array} {ccc} 1 & 0 & 0 \\0 & 1 & 0  \end{array}  \right],
\end{eqnarray}
and our goal is to solve 
\begin{eqnarray} \label{EO2N-7}
&\min&  t \nonumber \\
&s.t.&          U U^\top  \preceq tI, \nonumber \\
&&           U=U_0+x_1U_1+x_2U_2.
\end{eqnarray}
Then the input to DDS is defined as
\begin{eqnarray*} \label{EO2N-8}
&& \tx{cons\{1,1\}='MN'}, \ \ \tx{cons\{2,1\}}=[2 \ \ 3], \\
&& \tx{A\{1,1\}}=\left [\begin{array} {ccc} \tx{m2vec}(U_1,3) & \tx{m2vec}(U_2,3) &  \tx{zeros}(6,1) \\  \tx{zeros}(4,1) & \tx{zeros}(4,1) &  \tx{sm2vec}(I_{2\times 2})  \end{array}  \right], \ \ \tx{b\{1,1\}}=\left [\begin{array} {c} \tx{m2vec}(U_0,3) \\  \tx{zeros}(4,1)  \end{array}  \right],  \nonumber \\
&& \tx{c}=[0,0,1].
\end{eqnarray*}
\end{example} 

\subsection{Minimizing nuclear norm}
 The nuclear norm of a matrix $Z$ is $\|Z\|_*:=\trace\left ((ZZ^\top) ^{1/2} \right)$. The dual norm of $\|\cdot\|_*$ is the operator 2-norm $\|\cdot\|$ of a matrix. 
Minimization of nuclear norm has application in machine learning and matrix sparsification.  
 The following optimization problems are a primal-dual pair \cite{recht2010guaranteed}. 
\begin{eqnarray} \label{eq:dual-norm-1}
\begin{array} {ccc}  
(P_N) & \min_{X}  & \|X\|_* \\
& s.t. & A(X)=b.
\end{array} \ \ \ 
\begin{array} {ccc}  
(D_N) & \max_{z}  & \langle b,z \rangle \\
& s.t. & \|A^*(z)\| \leq 1,
\end{array}
\end{eqnarray}
where $A$ is a linear transformation on matrices and $A^*$ is its adjoint. $(P_N)$ is a very popular relaxation of the problem of minimizing $\rank(X)$ subject to $A(X)=b$, with applications in  machine learning and compressed sensing. The dual problem $(D_N)$ is a special case of \eqref{EO2N-1} where $Z=I$ and $U=A^*(z)$. As we will show on an example, solving $(D_N)$ by \tx{[x,y]=DDS(c,A,b,Z)} leads us to $y$, which gives a solution for $(P_N)$.

More specifically, consider  the optimization problem 
\begin{eqnarray} \label{EO2N-21}
&\min&  \|X \|_* \nonumber \\
&s.t.&      \trace(U_iX)=c_i, \ \ \ i \in \{1,\ldots,\ell \},
\end{eqnarray}
where $X$ is $n$-by-$m$. DDS can solve the dual problem by defining
\begin{eqnarray} \label{EO2N-22}
&& \tx{A\{1,1\}}=\left [\begin{array} {ccc} \tx{m2vec}(U_1,n) & \cdots & \tx{m2vec}(U_\ell,n) \\  \tx{zeros}(m^2,1) &  \cdots & \tx{zeros}(m^2,1)  \end{array}  \right], \ \ \tx{b\{1,1\}}=\left [\begin{array} {c} \tx{zeros}(mn,1)  \\  \tx{sm2vec}(I_{m\times m})  \end{array}  \right],  \nonumber \\
&& \tx{cons\{1,1\}='MN'}, \ \ \tx{cons\{1,2\}}=[m \ \ n]. 
\end{eqnarray}
Then, if we run \tx{[x,y]=DDS(c,A,b,cons)} and define \tx{V:=(vec2m(y\{1\}(1:m*n),m))$^\top$}, then $V$ is an optimal solution for \eqref{EO2N-21}. In subsection \ref{subsec:nn}, we present numerical results for solving problem \eqref{EO2N-21} and show that for the cases that $n \gg m$, DDS can be more efficient than SDP based solvers. 
Here is an example:

\begin{example}
We consider minimizing the nuclear norm over a subspace. Consider the following optimization problem:
\begin{eqnarray} \label{EO2N-16}
&\min&  \|X \|_* \nonumber \\
&s.t.&      \trace(U_1X)=1  \nonumber \\
&&          \trace (U_2X)=2,
\end{eqnarray}
where
\begin{eqnarray} \label{EO2N-17}
U_1=\left [\begin{array} {cccc} 1 & 0 & 0 & 0\\0 & 1 & 0 & 0  \end{array}  \right], \ \ U_2=\left [\begin{array} {cccc} 0 & 0 & 1 & 0 \\ 0 & 0 & 0 &1  \end{array}  \right].
\end{eqnarray}
By using \eqref{eq:dual-norm-1}, the dual of this problem is
\begin{eqnarray} \label{EO2N-18}
&\min&  -u_1-2u_2 \nonumber \\
&s.t.&      \|u_1 U_1 + u_2 U_2\| \leq 1. 
\end{eqnarray}
To solve this problem with our code, we define
\begin{eqnarray*} \label{EO2N-19}
&& \tx{cons\{1,1\}='MN'}, \ \ \tx{cons\{1,2\}}=[2 \ \ 4], \\
&& \tx{A\{1,1\}}=\left [\begin{array} {cc} \tx{m2vec}(U_1,4) & \tx{m2vec}(U_2,4) \\  \tx{zeros}(4,1) & \tx{zeros}(4,1)  \end{array}  \right], \ \ \tx{b\{1,1\}}=\left [\begin{array} {c} \tx{zeros}(8,1)  \\  \tx{sm2vec}(I_{2\times 2})  \end{array}  \right],  \nonumber \\
&& \tx{c}=[-1,-2]. 
\end{eqnarray*}
If we solve the problem using \tx{[x,y]=DDS(c,A,b,cons)}, the optimal value is $-2.2360$. Now \tx{V:=(vec2m(y\{1\}(1:8),2))$^\top$} is the solution of \eqref{EO2N-16} with objective value $2.2360$. We have
\begin{eqnarray} \label{EO2N-20}
X^*:=V=\left [\begin{array} {cc} 0.5 & 0 \\  0 & 0.5 \\ 1 & 0 \\ 0 & 1  \end{array}  \right].
\end{eqnarray}
\end{example}

\section{Epigraphs of convex univariate functions (geometric, entropy, and $p$-norm programming)}  
DDS accepts constraints of the form 
\begin{eqnarray} \label{intro-3}
\sum_{i=1}^\ell \alpha_i f_i(a_i^\top x + \beta_i) + g^\top x + \gamma  \leq 0,  \ \ \ a_i, g \in \mathbb R^{n}, \ \ \beta_i, \gamma  \in \mathbb R,  \ \ i \in  \{1,\ldots,\ell\},
\end{eqnarray}
where $\alpha_i \geq 0$ and  $f_i(x)$, $i \in  \{1,\ldots,\ell\}$, can be any function from Table \ref{table1}. Note that every univariate convex function can be added to this table  in the same fashion.  
By using this simple structure, we can model many interesting optimization problems. Geometric programming (GP) \cite{boyd2007tutorial} and entropy programming (EP) \cite{fang2012entropy} with many applications in engineering are constructed with constraints of the form \eqref{intro-3} when $f_i(z)=e^z$ for $i\in\{1,\ldots,\ell \}$ and $f_i(z)=z\ln(z)$ for $i\in\{1,\ldots,\ell \}$, respectively.  The other functions with $p$ powers let us solve optimization problems related to $p$-norm minimization. 
\begin{table} [h] 
  \caption{Some 2-dimensional convex sets and their s.c.\ barriers.}
  \label{table1}
  \center
  \renewcommand*{\arraystretch}{1.3}
  \begin{tabular}{ |c | c | c | }
    \hline
     & set $(z,t)$ & {s.c.\ barrier $\Phi(z,t)$} \\ \hline
    1& $-\ln(z) \leq t, \ z>0$ & $-\ln(t+\ln(z))-\ln(z)$ \\ \hline
    2 & $ e^z \leq t$ & $-\ln(\ln(t)-z)-\ln(t)$ \\ \hline
    3 & $z \ln(z) \leq t, \ z>0$ & $-\ln(t-z\ln(z)) - \ln(z)$ \\ \hline 
    4 & $|z|^p \leq t, \ p \geq 1$ & $-\ln(t^{\frac 2p} - z^2) - 2\ln(t)$ \\ \hline
    5 & $-z^p \leq t, \ z>0, \ 0 \leq p \leq 1$ & $-\ln(z^p+t)-\ln(z)$ \\ \hline 
    6 & $\frac 1z \leq  t, \ z>0$ & $-\ln(zt-1)$ \\ \hline
  \end{tabular}
\end{table}
 The corresponding s.c.\ barriers used in DDS are shown in Table \ref{table1}. There is a closed form expression for the LF conjugate of the first two functions. For the last four, the LF conjugate  can be calculated to high accuracy efficiently. In Appendix \ref{app:LF-univar}, we show how to calculate the LF conjugates for the functions in Table \ref{table1} and the internal functions we have in DDS. 

To represent a constraint of the from \eqref{intro-3}, for given $\gamma \in \R$ and $\beta_i \in \R$, $i \in \{1,\ldots,\ell\}$, we can define the corresponding convex set $D$ as  
\begin{eqnarray}
D:= \left \{  (w,s_i,u_i) :  w + \gamma \leq 0, \ \ f_i(s_i + \beta_i) \leq u_i, \ \forall i \right \},
\end{eqnarray}
and our matrix $A$ represents $w=\sum_{i=1}^\ell \alpha_i u_i + g^\top x$ and $s_i=a_i^Tx$, $i \in \{1,\ldots,\ell\}$. As can be seen, to represent our set as above, we introduce some auxiliary variables $u_i$'s to our formulations. 
DDS code does this internally. Let us assume that we want to add the following $s$ constraints to our model
\begin{eqnarray} \label{eq:TD-form}
\sum _{type} \sum_{i=1}^{\ell_{type}^j} \alpha_i^{j,type} f_{type}((a_i^{j,type})^\top x + \beta_i^{j,type}) + g_{j}^ \top x + \gamma_{j}   \leq 0,  \ \ \ \ j\in\{1,\ldots,s\}.
\end{eqnarray}
From now on, $type$ indexes the rows of Table \ref{table1}. 
The abbreviation we use for these constraints is TD. Hence, if the $k$th input block are the constraints in \eqref{eq:TD-form}, then we have \tx{cons\{k,1\}='TD'}. 
\tx{cons\{k,2\}}  is a {\bf cell array} of MATLAB with $s$ rows, each row represents one constraint. For the $j$th constraint we have:

\begin{itemize}
\item \tx{cons\{k,2\}\{j,1\}} is a matrix with two columns: the first column shows the type of a function from Table \ref{table1} and the second column shows the number of that function in the constraint. Let us say that in the $j$th constraint, we have
$l_{2}^j$ functions of type 2 and  $l_{3}^j$ functions of type 3, then we have
\begin{eqnarray*}
\tx{cons\{k,2\}\{j,1\}} =\left [\begin{array} {ccc}  2 & l_{2}^j  \\ 3  & l_3^j \end{array} \right ].
\end{eqnarray*}
The types can be in any order, but the functions of the same type are consecutive and the order must match with the rows of $A$ and $b$.

\item \tx{cons\{k,2\}\{j,2\}} is a vector with the coefficients of the functions in a constraint, i.e., $\alpha_i^{j,type}$ in \eqref{eq:TD-form}. Note that the coefficients must be in the same order as their corresponding rows in $A$ and $b$. If in the $j$th constraint we 
have 2 functions of type 2 and 1 function of type 3, it starts as
\begin{eqnarray*}
\tx{cons\{k,2\}\{j,2\}}=[\alpha_1^{j,2},  \alpha_2^{j,2}, \alpha_1^{j,3}, \cdots].
\end{eqnarray*}

\end{itemize}


To add the rows to $A$, for each constraint $j$, we first add $g_{j}$, then $a_i^{j,type}$'s in the order that matches \tx{cons\{k,2\}}.  We do the same thing for vector $b$ (first $\gamma_j$, then $\beta_i^{j,type}$'s). The part of $A$ and $b$ corresponding to the $j$th constraint is as follows if we have for example five types
\begin{eqnarray}
\tx{A}=\left [ \begin{array} {c} g_j^\top \\[.3em] a_1^{j,1} \\ \vdots \\  a_{l_1^j}^{j,1}\\ \vdots \\ a_1^{j,5} \\ \vdots \\  a_{l_5^j}^{j,5} \end{array}\right], \ \ \ \tx{b}=\left [ \begin{array} {c} \gamma_j \\[.3em]  \beta_1^{j,1} \\ \vdots \\  \beta_{l_1^j}^{j,1} \\ \vdots \\ \beta_1^{j,5} \\ \vdots \\  \beta_{l_5^j}^{j5} \end{array}\right].
\end{eqnarray}
Let us see an example:
\begin{example} Assume that we want to solve
\begin{eqnarray}
&\min&  c^\top x  \nonumber \\
&\text{s.t.}&  -1.2\ln(x_2+2x_3+55) + 1.8e^{x_1+x_2+1} + x_1 -2.1 \leq 0,  \nonumber \\
&&                -3.5\ln(x_1+2x_2+3x_3-30) + 0.9e^{-x_3-3} -x_3 +1.2 \leq 0, \nonumber \\
&&                 x \geq 0.
\end{eqnarray}
For this problem, we define:
\begin{eqnarray*}
&&\tx{cons\{1,1\}='LP'}, \ \ \tx{cons\{1,2\}}=[3], \\ 
&& \tx{cons\{2,1\}='TD'}, \ \   
\tx{cons\{2,2\}} = \left\{ \left[\begin{array} {ccc} 1 & 1 \\ 2 & 1\end{array} \right ], [1.2 \ \ 1.8] \ ; \ \  \left[\begin{array} {ccc} 1 &1 \\  2 & 1\end{array} \right ],  [3.5 \ \ 0.9] \right\}, \\
&&\tx{A\{1,1\}}=\left [ \begin{array} {ccc}-1& 0& 0 \\ 0 & -1 & 0 \\0 & 0 & -1 \end{array}\right], \ \ \ 
\tx{b\{1,1\}}=\left [ \begin{array} {c}  0 \\ 0 \\ 0 \end{array}\right], \\
&&\tx{A\{2,1\}}=\left [ \begin{array} {ccc} 1 & 0 & 0\\  0 & 1 & 2 \\  1 & 1 & 0 \\  0 & 0 & -1\\ 1 & 2 & 3\\  0 & 0 & -1 \end{array}\right], \ \ \ 
\tx{b\{2,1\}}=\left [ \begin{array} {c}  -2.1 \\ 55 \\ 1 \\ 1.2  \\  -30 \\ -3 \end{array}\right].
\end{eqnarray*}

\noindent {\bf Note:} As we mentioned, modeling systems for convex optimization that are based on SDP solvers, such as CVX, have to use approximation for functions involving $exp$ and $\ln$. Approximation makes it hard to return dual certificates, specifically when the problem is infeasible or unbounded.

%
\end{example}
\subsection{Constraints involving power functions}
The difference between these two types (4 and 5) and the others is that we also need to give the value of $p$ for each function. To do that, we add another column to \tx{cons\{k,2\}}. 

{\bf Note:} For TD constraints, \tx{cons\{k,2\}} can have two or three columns. If we do not use types 4 and 5, it has two, otherwise three columns. \tx{cons\{k,2\}\{j,3\}} is a vector which contains the powers $p$ for functions of types 4 and 5. The powers are given in the same order as the coefficients in  \tx{cons\{k,2\}\{j,2\}}. If the constraint also has functions of other types, we must put 0 in place of the power. 

Let us see an example:
\begin{example}
\begin{eqnarray*}
&\min&  c^\top x  \nonumber \\
&\text{s.t.}&  2.2\exp(2x_1+3)+|x_1+x_2+x_3| ^2 + 4.5 |x_1 + x_2| ^ {2.5} + |x_2 + 2x_3|^3 + 1.3x_1 -1.9  \leq 0.
\end{eqnarray*}
For this problem, we define:
\begin{eqnarray*}
&&\tx{A\{1,1\}}=\left [ \begin{array} {ccc}1.3& 0& 0 \\ 2 & 0 & 0 \\1 & 1 & 1 \\ 1 & 1 & 0\\  0 & 1 & 2  \end{array}\right], \ \ \ 
\tx{b\{1,1\}}=\left [ \begin{array} {c}  -1.9 \\ 3 \\ 0 \\ 0 \\ 0  \end{array}\right],  \nonumber \\
&&\tx{cons\{1,1\}='TD'}, \ \  \tx{cons\{1,2\}}=\left\{ \left [\begin{array} {ccc} 2 & 1 \\ 4 & 3\end{array} \right ], \  [2.2 \ \ 1 \ \ 4.5 \ \ 1], \ [0 \ \ 2 \ \ 2.5 \ \ 3] \right\}.
\end{eqnarray*}

\end{example}

\section{Vector Relative Entropy} \label{sec:RE}
Consider the relative entropy function $f: \R_{++}^\ell \oplus \R_{++}^\ell \rightarrow \R$ defined as 
\[
f(u,z):= \sum_{i=1}^{\ell} u_i\ln(u_i) - u_i\ln(z_i).
\] 
The epigraph of this function accepts the following $(2\ell+1)$-s.c.\ barrier (see Appendix \ref{sec:vRE}):
\begin{eqnarray} \label{eq:RE-1}
\Phi(t,u,z) := -\ln \left(t-\sum_{i=1}^{\ell} u_i\ln(u_i / z_i)\right)- \sum_{i=1}^{\ell}\ln(u_i)-\sum_{i=1}^{\ell}\ln(z_i).
\end{eqnarray}
This function covers many applications as discussed in \cite{dahl2019primal}. The special case of $\ell =1$ is used for handling EXP cone in the commercial solver MOSEK. We claim that the LF of $\Phi$ can be efficiently calculated using $\ell$ one-dimensional minimizations, i.e., having the function $\theta(r)$ as the solution of $\frac{1}{\theta} - \ln(\theta) = r$.  We consider the case of $\ell = 1$ and the generalization is straightforward. We have
\begin{eqnarray} \label{eq:RE-3}
\ \ \ \ \Phi_*(\alpha,y_u,y_z):= \min_{t,z,u} \alpha t + y_z z + y_u u +\ln(t-u\ln(u)+u\ln(z))+\ln(z)+\ln(u).
\end{eqnarray}
Writing the optimality  conditions, we get (defining $T:= t-u\ln(u)+u\ln(z)$)
\begin{eqnarray} \label{eq:RE-32}
\begin{array}{lcl}
\alpha + \frac{1}{T} &=& 0 \\
y_z+\frac{u}{z T} + \frac{1}{z} &=& 0 \\
y_u-\frac{1+\ln(u)-\ln(z)}{T} + \frac{1}{u}   &=& 0
\end{array} 
\end{eqnarray}
We can see that at the optimal solution, $T = -1/\alpha$, and if we calculate $u$, we can easily get $z$. $u$ is calculated by the following equation:
\begin{eqnarray}
\frac{1}{T}+\frac{1}{u} = \frac{1}{\theta \left(-Ty_u+2+\ln(-y_z) + \ln(T) \right)}
\end{eqnarray}
Therefore, to calculate $\Phi_*$ at every point, we need to evaluate $\theta$ once. For the general $\ell$, we evaluate $\theta$ by $\ell$ times. 
 
The abbreviation we use for relative entropy is RE. So, for the $k$th block of $s$ constraints of the form 
\begin{eqnarray} \label{eq:RE-2}
f(A_u^i x + b_u^i, A_z^i x + b_z^i) + g_i^\top x + \gamma_i  \leq 0,  \ \ i \in  \{1,\ldots, s\},
\end{eqnarray}
we define \tx{cons\{k,1\} = 'RE'} and \tx{cons\{k,2\}} is a vector of length $s$ with the $i$th element equal to $2\ell+1$. We also define:
\begin{eqnarray}
\tx{A}=\left [ \begin{array} {c} g_i^\top \\[.3em] A_u^1 \\ A_z^1 \\ \vdots \\  g_s^\top \\[.3em] A_u^s  \\ A_z^s \end{array}\right], \ \ \ \tx{b}=\left [ \begin{array} {c} \gamma_1 \\[.3em]  b_u^1 \\ b_z^1 \\ \vdots \\  \gamma_s \\ b_u^s \\ b_z^s \end{array}\right].
\end{eqnarray}

\begin{example} Assume that we want to minimize a relative entropy function under a linear constraint:
\begin{eqnarray}
\min  &  (0.8x_1+1.3) \ln \left( \frac{0.8x_1+1.3}{2.1x_1+1.3x_2+1.9}\right) + (1.1x_1-1.5x_2-3.8) \ln \left( \frac{1.1x_1-1.5x_2-3.8}{3.9x_2}\right)  \nonumber \\
s.t.  & x_1 + x_2 \leq \beta. 
\end{eqnarray}
We add an auxiliary variable $x_3$ to model the objective function as a constraint. 
For this problem we define:
\begin{eqnarray*}
&&\tx{cons\{1,1\}='RE'}, \ \  \tx{cons\{1,2\}}=\left [5 \right ] \\
&&\tx{A\{1,1\}}=\left [ \begin{array} {ccc}0& 0& -1 \\0.8 & 0 & 0 \\  1.1 & -1.5 & 0\\ 2.1 & 1.3 & 0 \\ 0 & 3.9 & 0  \end{array}\right], \ \ \ 
\tx{b\{1,1\}}=\left [ \begin{array} {c}  0  \\ 1.3 \\ -3.8 \\ 1.9 \\ 0 \end{array}\right],  \nonumber \\
&&\tx{cons\{2,1\}='LP'}, \ \  \tx{cons\{2,2\}}=\left [1 \right ] \\
&&\tx{A\{2,1\}}=[-1 \ \ -1 \ \ 0], \ \ \ \tx{b\{2,1\}}=[\beta].
\end{eqnarray*}
If we solve this problem by DDS, for $\beta = 2$ the problem is infeasible, and for $\beta = 7$ it returns an optimal solution $x^*:=(5.93,1.06)^\top$ with the minimum of function equal to $-7.259$. 
\end{example}
\section{Quantum entropy and Quantum relative entropy} \label{sec:QE}
Quantum entropy and quantum relative entropy functions are important in quantum information processing. DDS 2.0 accepts constraints involving these two functions. Let us start with the main definitions. 
Consider a function $f: \mathbb R \rightarrow \mathbb R\cup\{+\infty\}$ and let $X\in \mathbb H^n$ be a Hermitian matrix (with entries from $\mathbb C$) with a spectral decomposition $X=U \Diag(\lambda_1,\ldots,\lambda_n) U^*$, where $\Diag$ returns a diagonal matrix with the given entries on its diagonal and $U^*$ is the conjugate transpose of a unitary matrix $U$. We define the \emph{matrix extension} $F$ of $f$ as  $F(X) := U \Diag(f(\lambda_1),\ldots,f(\lambda_n)) U^*$. Whenever we write $\phi(X):=\trace(F(X))$, we mean a function $\phi: \mathbb H^n \rightarrow \R\cup\{+\infty\}$ defined as
\begin{eqnarray}\label{eq:fun_cal_1}
\phi(X):= \left\{ \begin{array}{ll}
\trace(U \Diag(f(\lambda_1),\ldots,f(\lambda_n)) U^*)& \text{if $f(\lambda_i) \in \R, \ \forall i$}, \\
+\infty  &  \text{o.w.} 
\end{array} \right.
\end{eqnarray} 
Study of such matrix functions go back to the work of L\" owner as well as Von-Neumann (see \cite{davis1957all}, \cite{lewis2003mathematics}, and the references therein).
A function $f: (a,b) \mapsto \mathbb R$ is said to be matrix monotone if for any two self-adjoint matrix $X$ and $Y$ with eigenvalues in $(a,b)$ that satisfy $X \succeq Y$, we have $F(X) \succeq F(Y)$. 
A function $f: (a,b) \mapsto \mathbb R$ is said to be \emph{matrix convex} if for any pair of self-adjoint matrices $X$ and $Y$ with eigenvalues in $(a,b)$, we have
\begin{eqnarray}\label{eq:fun_cal_2}
F(tX+(1-t)Y) \preceq tF(X)+(1-t)F(Y), \ \ \forall t\in (0,1). 
\end{eqnarray}
Faybusovich and Tsuchiya \cite{faybusovich2014matrix} utilized the connection between the matrix monotone functions and self-concordant functions. Let $f$ be a continuously differentiable function whose derivative is matrix monotone on the positive semi-axis and let us define $\phi$ as \eqref{eq:fun_cal_1}. Then, the function 
\begin{eqnarray}\label{eq:fun_cal_3}
\Phi(t,X):=-\ln(t-\phi(X))-\ln \det(X)
\end{eqnarray}
is a $(n+1)$-s.c. barrier for the epigraph of $\phi(X)$. This convex set has many applications. Many optimization problems arising in quantum information theory and some other areas requires dealing with the so-called \emph{quantum} or \emph{von Neumann entropy} $qe: \mathbb H^n \rightarrow \R\cup\{+\infty\}$ which is defined as $qe(X):=\trace(X \ln(X))$. If we consider $f(x):=x\ln(x)$, then $f'(x)=1+\ln(x)$ is matrix monotone on $(0,\infty)$ (see, for instance \cite{hiai2014introduction}-Example 4.2) and so we have a s.c. barrier for the set  
\[
\{(t,X) \in \R\oplus \mathbb S^n:  \trace(X\ln(X)) \leq t\}. 
\]
We have to solve the optimization problem
\begin{eqnarray}\label{eq:fun_cal_4}
\Phi_*(\eta,Y)=\sup_{t,X} t\eta+\langle X,Y \rangle +\ln(t-qe(X))+\ln \det(X),
\end{eqnarray}
to calculate the LF conjugate of \eqref{eq:fun_cal_3}, which is done in Appendix \ref{appen:quantum}.  

Another interesting function with applications in quantum information theory is \emph{quantum relative entropy} $qre: \mathbb H^n \oplus \mathbb H^n \rightarrow \R\cup\{+\infty\}$ defined as 
\[
qre (X,Y):= X\ln(X) - X\ln(Y).
\]
This function is convex as proved in \cite{tropp2015introduction}. The epigraph of $qre$ is
\[
\{(t,X,Y) \in \R\oplus \mathbb S^n \oplus \mathbb S^n:  \trace(X\ln(X)-X\ln(Y)) \leq t\}. 
\]
DDS 2.0 uses the following barrier (not yet known to be s.c.)\ for solving problems involving quantum relative entropy constraints:
\[
\Phi(t,X,Y):= \ln(t - qre(X,Y)) - \ln \det(X) - \ln \det(Y). 
\]

We can input constraints involving quantum entropy and quantum relative entropy into DDS as we explain in the following subsections. Appendix \ref{appen:quantum} contains some interesting theoretical results about quantum entropy functions and how to calculate the derivative and Hessian for the above s.c.\ barriers and also the LF conjugate given in \eqref{eq:fun_cal_4}. 

\subsection{Adding quantum entropy based constraints}
Let $qe_i : \symm^{n_i} \rightarrow \R \cup \{+\infty\}$ be quantum entropy functions and 
consider $\ell$ quantum entropy constraints of the form
\begin{eqnarray} \label{eq:QE-1}
  qe_i(A^i_0+x_1 A^i_1+ \cdots+x_n A^i_n) \leq  g_i^\top x+d_i, \ \ \ i\in\{1,\ldots,\ell\}.
\end{eqnarray}
$A^i_j$'s are $n_i$-by-$n_i$ symmetric matrices. 
To input \eqref{eq:QE-1} to DDS as the $k$th block, we define:
\begin{eqnarray} \label{eq:QE-4} 
&&\tx{cons\{k,1\}='QE'}, \ \ \tx{cons\{k,2\}}=[n_1, \ldots,n_\ell], \nonumber \\
&&\tx{A\{k,1\}}:=\left [\begin{array}{c} g_1^\top \\ \tx{sm2vec}(A^1_1), \cdots, \tx{sm2vec}(A^1_n)  \\ \vdots \\ g_\ell^\top \\ \tx{sm2vec}(A^\ell_1), \cdots, \tx{sm2vec}(A^\ell_n)\end{array} \right ], \ \ \ \tx{b\{k,1\}}:=\left [ \begin{array}{c} d_1 \\ \tx{sm2vec}(A^1_0)\\ \vdots \\ d_\ell \\ \tx{sm2vec}(A^\ell_0) \end{array} \right].
\end{eqnarray}

\begin{example}Assume that we want to find scalars $x_1$, $x_2$, and $x_3$ such that $2x_1+3x_2-x_3 \leq 5$ and  all the eigenvalues of  $H:=x_1A_1+x_2A_2+x_3A_3$ are at least 3, for 
\begin{eqnarray} \nonumber 
A_1=\left [ \begin{array}{ccc}1& 0& 0 \\ 0 & 1 & 0 \\0 & 0 & 1 \end{array} \right], \ 
A_2=\left [ \begin{array}{ccc}0& 0& 1 \\ 0 & 1 & 0 \\1 & 0 & 0 \end{array} \right], \
A_3=\left [ \begin{array}{ccc}0& 1& 0 \\ 1 & 0 & 0 \\0 & 0 & 0 \end{array} \right],
\end{eqnarray}
such that the quantum entropy $H$ is minimized.
We can write this problem as 
\begin{eqnarray}
&\min& t \nonumber \\
&s.t.& qe(x_1A_1+x_2A_2+x_3A_3) \leq t, \nonumber \\
&&  2x_1+3x_2-x_3 \leq 5, \nonumber \\
&& x_1A_1+x_2A_2+x_3A_3  \succeq 3I. 
\end{eqnarray} 
For the objective function we have $\tx{c}=(0,0,0,1)^\top$. Assume that the first and second blocks are LP and SDP as before. We define the third block of constraints as:
\begin{eqnarray*}
\tx{cons\{3,1\}='QE'}, \ \ \tx{cons\{3,2\}}=[3], \ \  \tx{b\{3,1\}}:=\left[\begin{array}{c}
0 \\ \tx{zeros}(9,1)
\end{array}	\right],  \\
\tx{A\{3,1\}}:=\left[\begin{array}{cccc} 0&0&0&1\\ \tx{sm2vec}(A1) & \tx{sm2vec}(A2) & \tx{sm2vec}(A3) & \tx{sm2vec}(\tx{zeros}(3))
\end{array}	\right].
\end{eqnarray*}
If we run DDS, the answer we get is $(x_1,x_2,x_3)=(4,-1,0)$ with $f(H)=14.63$. 
\end{example}

\subsection{Adding quantum relative entropy based constraints}
The abbreviation we use for quantum relative entropy is QRE. 
Let $qre_i : \symm^{n_i}\oplus \symm^{n_i} \rightarrow \R \cup \{+\infty\}$ be quantum relative entropy functions and 
consider $\ell$ quantum entropy constraints of the form
\begin{eqnarray*} \label{eq:QRE-1}
  qre_i(A^i_0+x_1 A^i_1+ \cdots+x_n A^i_n, B^i_0+x_1 B^i_1+ \cdots+x_n B^i_n) \leq  g_i^\top x+d_i, \ \ \ i\in\{1,\ldots,\ell\}.
\end{eqnarray*}
$A^i_j$'s and $B^i_j$'s are $n_i$-by-$n_i$ symmetric matrices. 
To input \eqref{eq:QE-1} to DDS as the $k$th block, we define:
\begin{eqnarray} \label{eq:QRE-4} 
&&\tx{cons\{k,1\}='QRE'}, \ \ \tx{cons\{k,2\}}=[n_1, \ldots,n_\ell], \nonumber \\
&&\tx{A\{k,1\}}:=\left [\begin{array}{c} g_1^\top \\ \tx{sm2vec}(A^1_1), \cdots, \tx{sm2vec}(A^1_n) \\ \tx{sm2vec}(B^1_1), \cdots, \tx{sm2vec}(B^1_n)  \\ \vdots \\ g_\ell^\top \\ \tx{sm2vec}(A^\ell_1), \cdots, \tx{sm2vec}(A^\ell_n) \\ \tx{sm2vec}(B^\ell_1), \cdots, \tx{sm2vec}(B^\ell_n) \end{array} \right ], \ \ \ \tx{b\{k,1\}}:=\left [ \begin{array}{c} d_1 \\ \tx{sm2vec}(A^1_0) \\ \tx{sm2vec}(B^1_0) \\ \vdots \\ d_\ell \\ \tx{sm2vec}(A^\ell_0) \\ \tx{sm2vec}(B^\ell_0) \end{array} \right].
\end{eqnarray}

\section{Hyperbolic polynomials} \label{sec:hyper}
Hyperbolic programming problems form one of the largest classes of convex optimization problems which have the kind of special mathematical structure amenable to more robust and efficient computational approaches. DDS 2.0 accepts hyperbolic constraints. Let us first define a hyperbolic polynomial. 
A polynomial $p(x)\in \R[x_1,\ldots,x_m]$ is said to be \emph{homogeneous} if every term has the same degree $d$. A homogeneous polynomial $p: \R^m \rightarrow \R$ is hyperbolic in direction $e \in \R^m$ if 
\begin{itemize}
\item $p(e)>0$. 
\item for every $x \in \R^m$, the univariate polynomial $p(x-te)$ has only real roots. 
\end{itemize}
Let $p(x)\in \R[x_1,\ldots,x_m]$ be hyperbolic in direction $e$. We define the \emph{eigenvalue function} $\lambda: \R^m \rightarrow \R^d$ with respect to $p$ and $e$ such that 
for every $x \in \R^m$, the elements of $\lambda(x)$ are the roots of the univariate polynomial $p(x-te)$. 
The underlying \emph{hyperbolicity cone} is defined as
\begin{eqnarray} \label{eq:hyper-1}
\Lambda_+(p,e):=\{x \in \R^m: \lambda(x) \geq 0\}. 
\end{eqnarray}

\begin{example}
The polynomial $p(x)=x_1^2-x_2^2-\cdots-x_m^2$ is hyperbolic in the direction $e:=(1,0,\ldots,0)^\top$ and the hyperbolicity cone with respect to $e$ is the second-order cone. The polynomial $p(X)=\det(X)$ defined on $\symm^n$ is hyperbolic in the direction $I$, and the hyperbolicity cone with respect to $I$ is the positive-semidefinite cone. 
\end{example}

By the above example, optimization over hyperbolicity cone is an extension of SOCP and SDP. The following theorem by G{\"u}ler gives a s.c.\ barrier for the hyperbolicity cone. 
 
\begin{theorem}[G{\"u}ler \cite{guler1997hyperbolic}] Let $p(x)$ be a homogeneous polynomial of degree $d$, which is hyperbolic in direction $e$. Then, the function $-\ln(p(x))$ is a $d$-LH s.c.\ barrier  for $\Lambda_+(p,e)$. 
\end{theorem}

DDS handles optimization problems involving hyperbolic polynomials using the above s.c.\ barrier. A computational problem is that, currently, we do not have a practical, efficient, algorithm to evaluate the LF conjugate of $-\ln(p(x))$. Therefore, DDS uses a primal-heavy version of the algorithm for these problems, as we explain in Section \ref{sec:heavy}. 

\subsection{Different formats for inputting multivariate polynomials} \label{sec:poly-format}
To input constraints involving hyperbolic polynomials, we use a matrix named \tx{poly}. In DDS, there are different options to input a multivariate polynomial:

 \ignore{A natural way is giving a polynomials by listing all the monomials. This representation my not be efficient. For example, the polynomial $p(x)=\det(X)$ for $X\in \symm^n$ has $n!$ monomials. It would be great if every hyperbolic polynomial had a determinantal representaion, but Br{\"a}nd{\'e}n proved \cite{branden2011obstructions} that there are  polynomials $p(x)$, hyperbolic in the direction $(1,0,\ldots,0)^\top$, such that for no power $k$ we have a representation as
\begin{eqnarray} \label{eq:hyper-2}
(p(x))^k=\det(x_1I+x_2 A_2+\cdots +x_mA_m),
\end{eqnarray}
where $A_i$'s are in symmetric matrices. }

\noindent{\bf Using monomials:} In this representation, if $p(x)$ is a polynomial of $m$ variables, then \tx{poly} is an $k$-by-$(m+1)$ matrix, where $k$ is the number of monomials. In the $j$th row, the first $m$ entries are the powers of the $m$ variables in the monomial, and the last entry is the coefficient of the monomial in $p(x)$. For example, if $p(x)=x_1^2-x_2^2-x_3^2$, then
\begin{eqnarray*}
\tx{poly}:=\left[\begin{array}{cccc} 2&0&0&1 \\ 0&2&0&-1 \\ 0&0&2&-1 
\end{array}	\right].
\end{eqnarray*}
{\bf Note:} In many applications, the above matrix is very sparse. DDS recommends  that in the monomial format, \tx{poly} should be defined as a sparse matrix. 

\noindent{\bf Using straight-line program:} Another way to represent a polynomial is by a \emph{straight-line program}, which can be seen as a rooted acyclic directed graph. The leaves represent the variables or constants. Each node is a simple binary operation (such as addition or multiplication), and the root is the result of the polynomial. In this case, \tx{poly} is a $k$-by-4 matrix, where each row represents a simple operation. Assume that $p(x)$ has $m$ variables, then we define
\[
f_0=1, \ \ f_i:=x_i, \ \ i\in\{1,\ldots,m\}.
\]
The $\ell$th row of \tx{poly} is of the form $[\alpha_j  \ \  i \ \  j \ \ \Box]$, which means that 
\[
f_{m+j}=\alpha_j (f_i \ \Box \ f_j).
\]
Operations are indexed by 2-digit numbers as the following table:
\begin{center}
\begin{tabular}{|c|c|} 
\hline
operation \ $\Box$& index \\ \hline
$+$  & 11 \\ \hline
$-$  & 22 \\ \hline
$\times$  & 33 \\ \hline
\end{tabular}
\end{center}
For example, if $p(x)=x_1^2-x_2^2-x_3^2$, we have the following representation: 
\begin{eqnarray*}
\left[\begin{array}{cccc} 1&1&1&33 \\ -1&2&2&33 \\ -1&3&3&33 \\ 1&4&5&11 \\ 1&6&7&11 
\end{array}	\right].
\end{eqnarray*}
Note that straight-line program is not unique for a polynomial.

\noindent{\bf Determinantal representation:} In this case, if possible, the polynomial $p(x)$ is written as
\begin{eqnarray} \label{eq:hyper-3}
p(x)=\det(H_0+x_1H_1+x_2 H_2+\cdots +x_mH_m),
\end{eqnarray}
where $H_i$, $i \in \{0,1,\ldots,m\}$ are in $\symm^m$.  We define
\[
\tx{poly}:=[\tx{sm2vec}(H_0) \ \ \tx{sm2vec}(H_1) \ \cdots \ \tx{sm2vec}(H_m)].
\]
For example, for $p(x)=x_1^2-x_2^2-x_3^2$, we can have
\[
H_0:=\begin{pmatrix} 0&0\\0&0\end{pmatrix}, \ \ 
H_1:=\begin{pmatrix} 1&0\\0&1\end{pmatrix}, \ \
H_2:=\begin{pmatrix} 1&0\\0&-1\end{pmatrix}, \ \
H_3:=\begin{pmatrix} 0&1\\1&0\end{pmatrix}. \ \
\]

\noindent {\bf Note:} $H_1,\ldots,H_m$ must be nonzero symmetric matrices. 
\subsection{Adding constraints involving hyperbolic polynomials}
Consider a hyperbolic polynomial constraint of the form
\begin{eqnarray} \label{eq:hyper-4}
p(Ax+b) \geq 0.
\end{eqnarray}
To input this constraint to DDS as the $k$th block, $A$ and $b$ are defined as before, and different parts of \tx{cons} are defined as follows:\\

\noindent \tx{cons\{k,1\}='HB'},\\
\noindent \tx{cons\{k,2\}=} number of variables in $p(z)$. \\
\noindent \tx{cons\{k,3\}} is the \tx{poly} that can be given in one of the three formats of Subsection \ref{sec:poly-format}. \\
\noindent \tx{cons\{k,4\}} is the format of polynomial that can be $\tx{'monomial'}$, $\tx{'straight\_line'}$, or $\tx{'determinant'}$. \\
\noindent \tx{cons\{k,5\}} is the direction of hyperbolicity or a vector in the interior of the hyperbolicity cone. 

\begin{example}
Assume that we want to give constraint \eqref{eq:hyper-4} to DDS for $p(x)=x_1^2-x_2^2-x_3^2$, using the monomial format. Then, \tx{cons} part is defined as
\begin{eqnarray*} \label{eq:hyper-5}
\tx{cons\{k,1\}='HB'}, \ \ \tx{cons\{k,2\}}=[3], \\
\tx{cons\{k,3\}}=\left[\begin{array}{cccc} 2&0&0&1 \\ 0&2&0&-1 \\ 0&0&2&-1 
\end{array}	\right], \\
\tx{cons\{k,4\}='monomial'}, \ \ \tx{cons\{k,5\}}=\left[\begin{array}{c} 1\\0\\0\end{array}\right].
\end{eqnarray*}
\end{example}

{
\section{Equality constraints}  \label{sec:eq}
In the main formulation of Domain-Driven form \eqref{main-p}, equality constraints are not included. However, the current version of DDS accepts equality constraints. In other form, users can input a problem of the form
\begin{eqnarray} \label{main-p-EQ}
\inf _{x} \{\langle c,x \rangle :  Bx=d, \ Ax \in D\}.
\end{eqnarray}

To input a block of equality constraints $Bx=d$, where $B$ is a $m_{EQ}$-by-$n$ matrix, as the $k$th block, we define
\[
\tx{cons\{k,1\}='EQ'}, \ \ \tx{cons\{k,2\}}=m_{EQ}, \ \ \tx{A\{k,1\}}=B, \ \ \tx{b\{k,1\}}=d. 
\]
\begin{example}
If for a given problem with $x \in \R^3$, we have a constraint $x_2-x_3=2$, then we can add it as the $k$th block by the following definitions:
\[
\tx{cons\{k,1\}='EQ'}, \ \ \tx{cons\{k,2\}}=1, \ \ \tx{A\{k,1\}}=[0 \ \ 1 \ -1], \ \ \tx{b\{k,1\}}=[2]. 
\]
\end{example}
In the main Domain-Driven formulation \eqref{main-p}, for some theoretical and practical considerations, we did not consider the equality constraints. 
From a mathematical point of view, this is not a problem. For any matrix $Z$ whose columns form a basis for the null space of $B$, we may represent the solution set of $Bx=d$ as $x^0+Zw$, where $x^0$ is any solution of $Bx=d$. Then the feasible region in \eqref{main-p-EQ} is equivalent to:
\begin{eqnarray}
D_w:=\{w: AZw \in (D-Ax^0)\}.
\end{eqnarray}
$D-Ax^0$ is a translation of $D$ with the s.c.\ barrier $\Phi(z-Ax^0)$, where $\Phi$ is a s.c.\ barrier for $D$. Therefore, theoretically, we can reformulate \eqref{main-p-EQ} as \eqref{main-p}, where  $AZ$ replaces $A$. 
Even though this procedure is straightforward in theory, there might be numerical challenges in application. For example, if we have a nice structure for $A$, such as sparsity, multiplying $A$ with $Z$ may ruin the structure. In DDS, we use a combination of heuristics, LU and QR factorizations to implement this transition efficiently. 

}

\section{Primal-heavy version of the algorithm} \label{sec:heavy}
For some class of problems, such as hyperbolic programming, a computable s.c.\ barrier $\Phi$ is available for the set $D$, while the LF of it is not available. For these classes, DDS uses a primal-heavy version of the algorithm. In the primal-heavy version, we approximate the primal-dual system of equations for computing the search directions by approximating the gradient and Hessian of $\Phi_*$. The approximations are based on the relations between the derivatives of $\Phi$ and $\Phi_*$: for every point $z \in \inte D$ we have
\begin{eqnarray}\label{eq:primalH-1}
z=\Phi_*'(\Phi'(z)), \ \ \ \Phi_*''(\Phi'(z))=[\Phi''(z)]^{-1}.
\end{eqnarray}
Instead of the primal-dual proximity measure defined in \cite{karimi_arxiv}, we use the primal-heavy version:
\begin{eqnarray}\label{eq:primalH-2}
\left\|\frac{\tau y}{\mu}-\Phi'(u)\right\|_{[\Phi''(u)]^{-1}},
\end{eqnarray}
where $u:=Ax+\frac{1}{\tau} z^0$, $\tau$ is an artificial variable we use in the formulation of the central path (see \cite{karimi_arxiv} for details), and $\mu$ is the parameter of the central path. By \cite{karimi_arxiv}-Corollary 4.1, this proximity measures is ``equivalent" to the primal-dual one, but \eqref{eq:primalH-2} is less  efficient computationally.

By using a primal-heavy version, we lose some of the desired properties of primal-dual setup, such as the ability to move in a wider neighbourhood of the central path. Moreover, in the primal-heavy version, we have to somehow make sure that  the dual iterates $y$ are feasible (or at least the final dual iterate is). Another issue is with calculating the duality gap \eqref{eq:duality-gap}. For a general convex domain $D$, we need $\Phi_*'$ to accurately calculate $\delta_*(y|D)$ as explained in \cite{karimi_arxiv}. Note that when $D$ is a shifted cone $D=K-b$, then we have
\begin{eqnarray}\label{eq:primalH-3}
	\delta_*(y|D)=-\langle b,y \rangle. 
\end{eqnarray}
To calculate the duality gap, we can write it as the summation of separate terms for the domains, and  
if a domain with only the primal barrier is a shifted cone, we can use \eqref{eq:primalH-3}. This is the case for the current version of DDS as all the set constraints using primal-heavy version are based on a shifted convex cone. 
To make sure that the dual iterates are feasible, we choose our neighborhoods to satisfy 
\begin{eqnarray}\label{eq:primalH-4}
\left\|\frac{\tau y}{\mu}-\Phi'(u)\right\|_{\Phi''(u)} < 1,
\end{eqnarray}
and by the Dikin ellipsoid property of s.c.\ functions, $y$ iterates stay feasible.
We can specify in OPTIONS if we want to use a primal-heavy version of the algorithm in DDS by 

\tx{OPTIONS.primal=1;}




\section{Numerical Results}  \label{sec:num}
In this section, we present some numerical examples of running DDS 2.0. We performed computational experiments using the software MATLAB R2018b, on a 4-core 3.2 GHz Intel Xeon X5672 machine with 96GB of memory. Many of the examples in this section are given as mat-files in the DDS package; in the folder \tx{Test\_Examples}.

DDS pre-processes the input data and returns the statistics before and after the pre-processing. This processing includes checking the structure of the input data removing obviously redundant constraints. For example, if the size of $A$ given for a block of constraints does not match the information in \tx{cons}, DDS returns an error. The pre-processing also includes scaling the input data properly for each block of constraints. 

As mentioned in Section 2, users have the option to give the starting point of the infeasible-start algorithm as input. Otherwise, DDS chooses an easy-to-find starting point for every input set constraint. The stopping criteria in DDS are based on the status determination result in \cite{karimi_status_arxiv}, which studied how to efficiently certify different status of a given problem in Domain-Driven form, using duality theory.  \cite{karimi_status_arxiv}-Section 9 discusses stopping criteria in a practical setting. We define the following parameters for the current point $(x,\tau,y)$:
\begin{eqnarray} \label{eq:stop-4}
gap:=\frac{|\langle c,x\rangle+\frac{1}{\tau}\delta_*(y|D)|}{1+|\langle c,x\rangle|+|\frac{1}{\tau}\delta_*(y|D)|}, \ \ P_{feas}:= \frac{1}{\tau} \|z^0\|, \ \ D_{feas}:= \frac{\|\frac{1}{\tau} A^\top y + c\|}{1+\|c\|},
\end{eqnarray}
where $x$ and $y$ are the primal and dual points, and $\tau$ is the artificial variable we use (see \cite{karimi_arxiv} for details). DDS stops the algorithm and returns the status as ``solved" when we have
\begin{eqnarray} \label{eq:stop-5}
\max\{ gap, P_{feas}, D_{feas}\} \leq tol. 
\end{eqnarray}
DDS returns that status as ``infeasible" if the returned certificate $\bar y := \frac{\tau}{\mu} y$ satisfies
\begin{eqnarray} \label{eq:stop-6}
\|A^\top \bar y\| \leq tol, \ \ \  \delta_*(\bar y|D) < 0, 
\end{eqnarray}
and returns the status as ``unbounded"  if
\begin{eqnarray} \label{eq:stop-7}
\langle c, x \rangle \leq - \frac{1}{tol}. 
\end{eqnarray}
DDS has different criteria for returning ``ill-conditioned" status, which can be because of numerical inaccuracy issues. 

In the following, we see the numerical performance of DDS for different combinations of set constraints. 

%
%
%

\subsection{LP-SOCP-SDP}
In this subsection, we consider LP-SOCP-SDP instances mostly from the Dimacs library \cite{dimacs}. Note that the problems in the library are for the standard equality form and we solve the dual of the problems.
Table \ref{table:dimacs-1} shows the results. DDS has a built-in function \tx{read\_sedumi\_DDS} to read problems in SeDuMi
input format. You can use either of the following commands based on the SeDuMi format file

\begin{verbatim}
[c,A,b,cons]=read_sedumi_DDS(At,b,c,K);
[c,A,b,cons]=read_sedumi_DDS(A,b,c,K);
\end{verbatim}


\begin{table} [h] 
  \caption{Numerical results for some problem from the Dimacs library for $tol=10^{-8}$.}
  \label{table:dimacs-1}
  \renewcommand*{\arraystretch}{1.3}
  \begin{tabular}{ |c | c| c | c | c |  }
    \hline
Problem  &    size of $A$     &    Type of Constraints   &    Iterations   \\ \hline 

 nb  &    $2383*123$     &    SOCP-LP   &     42  \\  \hline 
 nb\_L1  &    $3176*915$     &    SOCP-LP   &     37    \\ \hline 
 nb\_L2  &    $4195*123$     &    SOCP-LP   &     23  \\  \hline 
 nb\_L2\_bessel  &    $2641*123$     &    SOCP-LP   &     27 \\  \hline 
 filter48\_socp  &    $3284*969 $    &    SDP-SOCP-LP   &     80  \\  \hline 
 filtinf1  &    $3395*983$     &    SDP-SOCP-LP   &     26   \\ \hline 
 truss5  &    $3301*208$     &       SDP   &     66   \\ \hline 
 truss8  &    $11914*496$     &       SDP   &     76   \\ \hline 
 copo14  &    $3108*1275$     &       SDP-LP   &     24    \\ \hline 
 copo23  &    $3108*1275$     &       SDP-LP   &     32   \\ \hline 
 toruspm3-8-50  &    $262144*512$     &       SDP   &     20    \\ \hline 
 torusg3-8  &    $262144*512$     &       SDP   &     24   \\ \hline 
 sched\_50\_50\_scaled  &    $4977*2526$     &       SOCP-LP   &     81   \\ \hline 
 mater-3  &    $39448*1439$     &    SDP-LP   &    124   \\  \hline 
 cnhil8  &    $14400*1716$     &    SDP   &     31   \\  \hline 
 cnhil10  &    $48400*5005$     &    SDP   &     35   \\  \hline 
 cphil10  &    $48400*5005$     &    SDP   &     12   \\  \hline 
 ros\_500  &    $17944*4988$     &       SDP   &     41    \\ \hline 
 sensor\_500   &    $245601*3540$    &       SDP   &     65    \\ \hline 
 taha1a  &    $231672*3002$     &       SDP   &     23      \\ \hline 
 taha1a  &    $231672*3002$     &       SDP   &     42    \\ \hline 
 G40mc  &    $4000000*2000$     &       SDP   &     33   \\ \hline 
 1tc.1024  &    $1048576*7937$     &       SDP   &     41   \\ \hline 
 yalsdp  &    $30000*5051$     &       SDP   &     17    \\ \hline 
 
\end{tabular}
\end{table}

\subsection{EXP cone optimization problems from CBLIB}
 Conic Benchmark Library (CBLIB) \cite{friberg2016cblib}  is a collection of benchmark problems for conic mixed-integer and continuous optimization. It contains a collection of optimization problems involving exponential cone, they show as EXP. Exponential cone is a special case of vector relative entropy we discussed in  Section \ref{sec:RE} when $\ell = 1$. Table \ref{table:cblib-exp} show the results of running DDS to solve problems with EXP cone from CBLIB. The files converted into DDS format can be found in the \tx{Test\_Examples} folder. 
\begin{table} [!h] 
  \caption{Numerical results for some EXP cone problems from CBLIB.}
  \label{table:cblib-exp}
  \renewcommand*{\arraystretch}{1.3}
  \begin{tabular}{ |c | c| c | c | c |  c|}
    \hline
Problem  &    size of $A$   &  size of $A_{EQ}$ &    Type of Constraints   &    Iterations   \\ \hline 
syn30m  &    $324*121$    &  $19*121$ &    EXP-LP   &     36  \\  \hline
syn30h  &    $528*249$    &  $161*249$ &    EXP-LP   &     51  \\  \hline  
syn30m02h  &    $1326*617$    &  $381*617$ &    EXP-LP   &     46  \\  \hline
syn05m02m &    $177*67$   &  $11*67$ & EXP-LP & 36 \\ \hline
syn10h &       $171*84$  & $55*84$  & EXP-LP & 37 \\ \hline
syn15m & $177*67$  & $12*67$ & EXP-LP & 31 \\ \hline
syn40h & $715*331$ & $213*331$ & EXP-LP & 48 \\ \hline
synthes1  &    $30*13$    &  $1*13$ &    EXP-LP   &     13  \\  \hline 
synthes2  &    $41*16$    &  $1*12$ &    EXP-LP   &     26  \\  \hline 
synthes3 &  $71*26$  & $3*26$ & EXP-LP & 20 \\ \hline
ex1223a & $59*17$ & $1*17$ & EXP-LP  & 11 \\ \hline
gp\_dave\_1 &    $985*705$    &  $453*705$ &    EXP-LP   &     37  \\  \hline
fiac81a &    $163*191$    &  $126*191$ &    EXP-LP   &     15  \\  \hline
fiac81b &    $65*87$    &  $57*87$ &    EXP-LP   &     15  \\  \hline
batch &    $175*58$    &  $13*58$ &    EXP-LP   &     160  \\  \hline
demb761 &    $93*131$    &  $90*131$ &    EXP-LP   &     24  \\  \hline
demb762 &    $93*131$    &  $90*131$ &    EXP-LP   &     27  \\  \hline
demb781 &    $14*19$    &  $12*19$ &    EXP-LP   &     8  \\  \hline
enpro48 &    $506*167$    &  $33*167$ &    EXP-LP   &     216  \\  \hline
isi101 &    $467*393$    &  $261*33$ &    EXP-LP   &     52 (infeas)  \\  \hline
jha88 &    $866*1131$    &  $825*1131$ &    EXP-LP   &     33  \\  \hline
LogExpCR-n20-m400 & $2023*2022$ & $1200*2022$ & EXP-LP & 40 \\ \hline
LogExpCR-n100-m400 & $2103*2102$ & $1200*2102$ & EXP-LP & 46 \\ \hline
LogExpCR-n20-m1200 & $6023*6022$ & $3600*6022$ & EXP-LP & 44 \\ \hline
LogExpCR-n500-m400 & $2503*2502$ & $1200*2502$ & EXP-LP & 68 \\ \hline
LogExpCR-n500-m1600 & $8503*8502$ & $4800*8502$ & EXP-LP & 63 \\ \hline
\end{tabular}
\end{table}
 
\subsection{Minimizing Nuclear Norm}  \label{subsec:nn}
In this subsection, we are going to use DDS to solve problem \eqref{EO2N-21}; minimizing nuclear norm of a matrix subject to linear constraints. We are interested in the case that $X$ is an $n$-by-$m$ matrix with $n \gg m$, which makes the s.c.\ barrier for the epigraph of a matrix norm more efficient than the SDP representation. We compare DDS with convex modeling system CVX, which has a function \tx{norm\_nuc} for nuclear norm. For numerical results, we assume \eqref{EO2N-21} has two linear constraints and $U_1$ and $U_2$ are 0-1 matrices generated randomly. Let us fix the number of columns $m$ and calculate the running time by changing the number of rows $n$. Figure \ref{Fig1} shows the plots of running time for both DDS and CVX. For epigraph of a matrix norm constraints, DDS does not form matrices of size $m+n$, and as can be seen in the figure, the running time is almost linear as a function of $n$. On the other hand, the CVX running time is clearly super-linear.

\begin{center}
	\begin{figure}
	\includegraphics[scale=0.7]{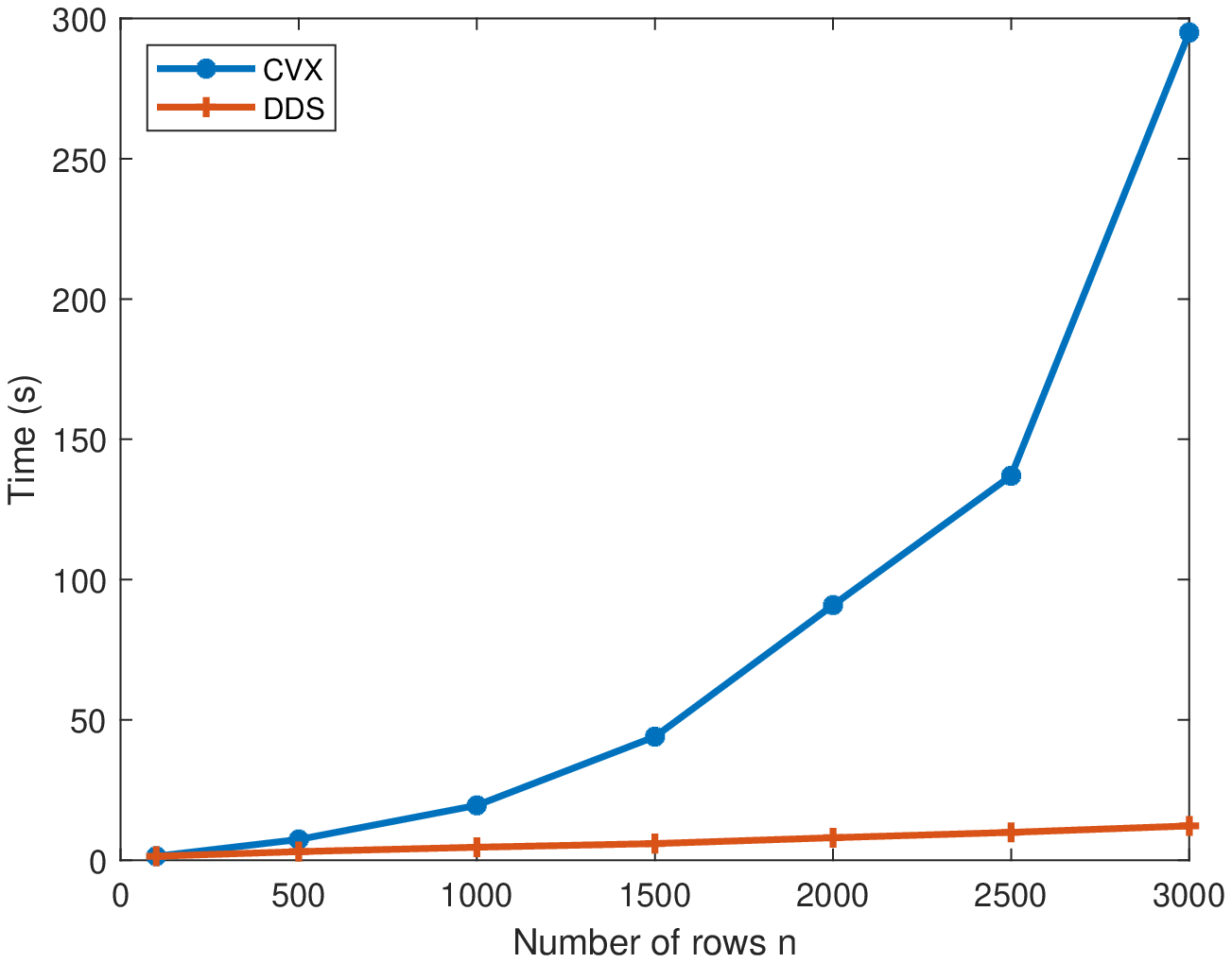}
	\caption{\small Average running  time for minimizing nuclear norm versus the number of rows in the matrix, where the number of columns is fixed at $m=20$. }
	\label{Fig1}
	\end{figure}
\end{center}

\subsection{Quantum Entropy and Quantum Relative Entropy}
In this subsection, we see optimization problems involving quantum entropy. As far as we know, there is no library for quantum entropy and quantum relative entropy functions. We have created examples of combinations of quantum entropy constraints and other types such as LP, SOCP, and $p$-norm. Problems for quantum entropy are of the form 
\begin{eqnarray} \label{eq:num-2}
&\min& t \nonumber \\
& \text{s.t.}& qe( A_0+x_1 A_1 + \cdots+x_n A_n) \leq t , \nonumber  \\
&&  \text{other constraints},
\end{eqnarray}
where $qe = \trace{X\ln(X)}$ and $A_i$'s are sparse symmetric matrices. The problems for $qre$ are also in the same format. Our collection contains both feasible and infeasible cases. The results are shown in Tables \ref{table:QE} and \ref{table:QRE} respectively for quantum entropy and quantum relative entropy. The MATLAB files in the DDS format for some of these problems are given in the \tx{Test\_Examples} folder. We compare our results to CVXQUAD, which is a collection of matrix functions to be used on top of CVX. The package is based on the paper  \cite{fawzi2019semidefinite}, which approximates matrix logarithm with functions that an be described as feasible region of SDPs. Note that these approximations increase the dimension of the problem, which is why CVX fails for the larger problems in Table \ref{table:QE}. 
\begin{table} [!h] 
  \caption{Results for problems involving Quantum Entropy optimization problems.}
  \label{table:QE}
  \renewcommand*{\arraystretch}{1.3}
  \begin{tabular}{ |c | c| c | c | c |  c| c|}
    \hline
Problem  &    size of $A$   &     Type of Constraints   &    Itr/ \ time(s)  &  CVXQUAD \\ \hline 
QuanEntr-10  & $101*21$  & QE & 11/\ 0.6  &  12/ \ 1.3  \\ \hline
QuanEntr-20  & $401*41$ & QE  & 13/\ 0.8 & 13/\ 2.4 \\ \hline
QuanEntr-50  & $2501*101$ & QE & 18/\ 1.8 &  12/\ 23.2 \\ \hline
QuanEntr-100  & $10001*201$ & QE & 24/\ 6.7  & array size error \\ \hline
QuanEntr-200  & $40001*401$ & QE & 32/\ 53.9 & array size error \\ \hline
QuanEntr-LP-10  & $111*21$  & QE-LP & 19/\ 1.1 & 16/\ 3.4 \\ \hline
QuanEntr-LP-20  & $421*41$ & QE-LP  & 23 /\ 1.6 & 20/\ 9.1 \\ \hline
QuanEntr-LP-50  & $2551*101$ & QE-LP & 34/\ 2.3 & 20/\ 103.4 \\ \hline
QuanEntr-LP-100  & $10101*201$ & QE-LP & 44/\ 14.5 & array size error\\ \hline
QuanEntr-LP-SOCP-10-infea  & $122*21$ & QE-LP-SOCP & 12/\ 0.6 (infea) & 16/\ 5.1 \\ \hline
QuanEntr-LP-SOCP-10  & $122*21$ & QE-LP-SOCP & 26/\ 1.0 & 24/\ 2.0 \\ \hline
QuanEntr-LP-SOCP-20-infea  & $441*41$ & QE-LP-SOCP & 14/\ 1.1 (infea) & 14/\ 5.8 \\ \hline
QuanEntr-LP-SOCP-20  & $441*41$ & QE-LP-SOCP & 33/\ 1.8 & 22/\ 4.1 \\ \hline
QuanEntr-LP-SOCP-100  & $10201*201$ & QE-LP-SOCP & 51/\ 20.8  & array size error  \\ \hline
QuanEntr-LP-SOCP-200  & $40402*401$ & QE-LP-SOCP & 69/\ 196  & array size error  \\ \hline
QuanEntr-LP-3norm-10-infea & $121*21$ & QE-LP-pnorm & 16/\ 0.6 (infea) & 22/\ 2.0\\ \hline
QuanEntr-LP-3norm-10 & $121*21$ & QE-LP-pnorm & 21/\ 1.2 & 19/\ 2.0 \\ \hline
QuanEntr-LP-3norm-20-infea & $441*41$ & QE-LP-pnorm & 16/\ 0.9 (infea) & 18/\ 6.5 \\ \hline
QuanEntr-LP-3norm-20 & $441*41$ & QE-LP-pnorm & 29/\ 1.1 & 20/\ 7.3 \\ \hline
QuanEntr-LP-3norm-100-infea & $10201*201$ & QE-LP-pnorm & 25/\ 15.3 (infea) &  array size error\\ \hline
QuanEntr-LP-3norm-100 & $10201*201$ & QE-LP-pnorm & 66/\ 40.5 &  array size error \\ \hline

\end{tabular}
\end{table}

\begin{table} [!h] 
  \caption{Results for problems involving Quantum Relative Entropy optimization problems.}
  \label{table:QRE}
  \renewcommand*{\arraystretch}{1.3}
  \begin{tabular}{ |c | c| c | c | c |  c| c|}
    \hline
Problem  &    size of $A$   &     Type of Constraints   &    Itr/ \ time(s)  &  CVXQUAD \\ \hline 
QuanReEntr-6  & $73*13$  & QRE & 12/\ 6.3  &  12/ \ 18.1  \\ \hline
QuanReEntr-10  & $201*21$  & QRE & 12/\  21.2  &  25/ \ 202  \\ \hline
QuanReEntr-20  & $801*41$  & QRE & 17/\  282  &  array size error  \\ \hline
QuanReEntr-LP-6  & $79*13$  & QRE-LP & 25/\ 23.2  &  21/ \ 19.6  \\ \hline
QuanReEntr-LP-6-infea  & $79*13$  & QRE-LP & 30/\ 11.8  &  28/ \ 21.5  \\ \hline
QuanReEntr-LP-10  & $101*21$  & QRE-LP & 26/\ 49.3  &  24/ \ 223  \\ \hline

\end{tabular}
\end{table}

\subsection{Hyperbolic Polynomials}
We have created a library of hyperbolic cone programming problems for our experiments. 
Hyperbolic polynomials are defined in Section \ref{sec:hyper}. We discuss three methods in this section for the creation of our library. 

\subsubsection{Hyperbolic polynomials from matriods} Let us first define a matroid.
\begin{definition}
A \emph{ground set} $E$ and a collection of \emph{independent sets } $\mathcal I \subseteq 2^E$ form a \emph{matroid} $M=(E,\mathcal I)$ if:
\begin{itemize}
\item $\emptyset \in \mathcal I$, 
\item if $A \in \mathcal I$ and $B \subset A$, then $B \in \mathcal I$, 
\item if $A,B \in \mathcal I$, and $|B| > |A|$, then there exists $b \in B \setminus A$ such that $A \cup \{b\} \in \mathcal I$. 
\end{itemize}
\end{definition}
The independent sets with maximum cardinality are called the bases of the matroid, we denote the set of bases of the matroid by $\mathcal B$. We can also assign a \emph{rank} function $r_M: 2^E \rightarrow \mathbb Z_+$ as:
\begin{eqnarray} \label{eq:hyper-7}
r_M(A) := \max \{|B| : B \subseteq A, B \in \mathcal I\}.
\end{eqnarray}
Naturally, the rank of a matroid is the cardinality of any basis of it.
The \emph{basis generating polynomial} of a matroid is defined as
\begin{eqnarray} \label{eq:hyper-6}
p_M(x) := \sum_{B \in \mathcal B}  \  \prod_{i \in B}  x_i.
\end{eqnarray}
A class of hyperbolic polynomials are the basis generating polynomials of certain matroids with the \emph{half-plane property} \cite{choe2004homogeneous}. A polynomials has the half-plane property if it is nonvanishing whenever all the variables lie in the open right half-plane. We state it as the following lemma:
\begin{lemma}
Assume that $M$ is a matroid with half-plane property. Then its basis generating polynomial is hyperbolic in any direction in the positive orthant. 
\end{lemma}

Several classes of matroids have been proven to have the half-plane property \cite{choe2004homogeneous, branden2007polynomials, AminiBranden2018}. As the first example, we introduce the Vamos matroid. The ground set has size 8 (can be seen as 8 vertices of the cube) and the rank of matroid is 4. All subsets of size 4 are independent except 5 of them.  Vamos matroid has the half-plane property \cite{wagner2009criterion}. The basis generating polynomial is:
\begin{eqnarray*} \label{eq:hyper-8}
p_V(x):= E_4(x_1,\ldots,x_8) - x_1x_2x_3x_4-x_1x_2x_5x_6-x_1x_2x_6x_8-x_3x_4x_5x_6-x_5x_6x_7x_8,
\end{eqnarray*}	
where $E_d(x_1,\ldots,x_m)$  is the elementary  symmetric polynomial of degree $d$ with $m$ variables. Note that elementary symmetric polynomials are hyperbolic in direction of the vector of all ones.  Some extensions of Vamos matroid also satisfy the half-plane property \cite{burton2014real}. These matroids give the following Vamos-like basis generating polynomials:
\begin{eqnarray} \label{eq:hyper-9}
p_{VL}(x):= E_4(x_1,\ldots,x_{2m}) - \left(\sum_{i=2}^m x_1x_2x_{2k-1}x_{2k} \right) - \left(\sum_{i=2}^{m-1} x_{2k-1}x_{2k}x_{2k+1}x_{2k+2}\right).
\end{eqnarray}
Vamos like polynomials in  \eqref{eq:hyper-9} have an interesting property of being counter examples to one generalization of 	Lax conjecture \cite{branden2011obstructions, burton2014real}. Explicitly, there is no power $k$ and symmetric matrices $H_2, \ldots, H_{2m}$ such that  
\begin{eqnarray} \label{eq:hyper-13}
(p_{VL}(x))^k = \det(x_1I+x_2H_2+\cdots+x_{2m}H_{2m}).
\end{eqnarray}
In the DDS package, we have created a function 

\tx{poly = vamos(m)}

\noindent that returns a Vamos like polynomial as in \eqref{eq:hyper-9} in the \tx{'monomial'} format, which can be given as \tx{cons\{k,3\}} for inputting a hyperbolic constraint. \tx{vamos(4)} returns the Vamos polynomial with 8 variables. 

Graphic matroids also have the half-plane property \cite{choe2004homogeneous}. However, the hyperbolicity cones arising
from graphic matroids are isomorphic to positive semidefinite cones. This can be
proved using a classical result that the characteristic polynomial
of the bases of graphic matroids is the determinant of the Laplacian
(with a row and a column removed) of the underlying graph \cite{branden2014hyperbolicity}.
Consider a graph $G=(V,E)$ and let $\mathcal T$ be the set of all spanning trees of $G$. Then the generating polynomial defined as 
\begin{eqnarray} \label{eq:hyper-10}
T_G(x) := \sum_{T \in \mathcal T}  \  \prod_{e \in T}  x_e
\end{eqnarray}
is a hyperbolic polynomial. Several matroid operations preserve the half-plane property as proved in \cite{choe2004homogeneous}, including taking minors,
duals, and direct sums.

We have created a set of problems with combinations of hyperbolic polynomial inequalities and other types of constraints such as those arising from LP, SOCP, and entropy function.  

\begin{table} [!h] 
  \caption{Results for problems involving hyperbolic polynomials.}
  \label{table:HB}
  \renewcommand*{\arraystretch}{1.3}
  \begin{tabular}{ |c | c| c | c | c |  c| c|}
    \hline
Problem  &    size of $A$   &    var/deg of $p(z)$ & Type of Constraints   &    Itr/ \ time(s) \\ \hline 
Vamos-1  & $8*4$  & 8/4  &  HB & 7/ \ 1.5    \\ \hline
Vamos-LP-1  & $12*4$  & 8/4  &  HB-LP & 8/ \ 2    \\ \hline
Vamos-SOCP-1  & $17*5$  & 8/4  &  HB-LP-SOCP & 12/ \ 2    \\ \hline
Vamos-Entr-1  & $17*5$  & 8/4  &  HB-LP-Entropy & 9/ \ 1.2    \\ \hline
VL-Entr-1  & $41*11$  & 20/4  &  HB-LP-Entropy & 9/ \ 24    \\ \hline
VL-Entr-1  & $61*16$  & 30/4  &  HB-LP-Entropy & 12/ \ 211    \\ \hline

\end{tabular}
\end{table}

\subsubsection{Derivatives of products of linear forms}
Hyperbolicity of a polynomial is preserved under directional derivative operation (see \cite{renegar2006hyperbolic}):
\begin{theorem}  \label{thm:hyper-1}
Let $p(x)\in \R[x_1,\ldots,x_m]$ be hyperbolic in direction $e$ and of degree at least 2. Then the polynomial 
\begin{eqnarray} \label{eq:hyper-11}
p'_e(x) := (\nabla p)(x)[e]
\end{eqnarray}
is also hyperbolic in direction $e$. Moreover, 
\begin{eqnarray} \label{eq:hyper-12}
\Lambda(p, e)  \subseteq \Lambda (p'_e, e).
\end{eqnarray}
\end{theorem}
Assume that $p(x):= l_1(x) \cdots l_\ell(x)$, where $l_1(x),\ldots,l_\ell(x)$ are linear forms. Then, $p$ is hyperbolic in the direction of any vector $e$ such that $p(e) \neq 0$. As an example, consider $E_m(x_1,\ldots,x_m) = x_1 \cdots x_m$. 
Recursively applying Theorem \ref{thm:hyper-1} to such polynomials $p$ leads to many hyperbolic polynomials, including all elementary symmetric polynomials. For some properties
of hyperbolicity cones of elementary symmetric polynomials, see \cite{Zinchenko2008}. For some preliminary computational experiments on such hyperbolic programming problems, see \cite{Myklebust2015}. 

For the polynomials constructed by the products of linear forms, it is more efficient to use the straight-line program. 
For a polynomial $p(x):= l_1(x) \cdots l_\ell(x)$, let us define a $\ell$-by-$m$ matrix $L$ where the $j$th row contains the coefficients of $l_j$. In DDS,  we have created a function 

 \tx{[poly,poly\_grad] = lin2stl(M,d)}

\noindent which returns two polynomials in "straight-line" form, one as the product of linear forms defined by $M$, and the other one its directional derivative in the direction of $d$. For example, if we want the polynomial $p(x):= (2x_1-x_3)(x_1-3x_2+4x_3)$, then our $M$ is defined as 
\begin{eqnarray} \label{eq:hyper-14}
M:= \left[\begin{array}{ccc}
2&0&-1 \\ 1&-3&4 
\end{array} \right].
\end{eqnarray}
Table \ref{table:HB2} shows some results of problem where the hyperbolic polynomial is either a product of linear forms or their derivatives. 
\begin{table} [!h] 
  \caption{Results for problems involving hyperbolic polynomials as product of linear forms and their derivatives.}
  \label{table:HB2}
  \renewcommand*{\arraystretch}{1.3}
  \begin{tabular}{ |c | c| c | c | c |  c| c|}
    \hline
Problem  &    size of $A$   &    var/deg of $p(z)$ & Type of Constraints   &    Itr \\ \hline 
HPLin-LP-1  & $55*50$  & 5/3  &  HB-LP & 8    \\ \hline
HPLinDer-LP-1  & $55*50$  & 5/3  &  HB-LP & 7    \\ \hline
HPLin-LP-2  & $55*50$  & 5/7  &  HB-LP & 19    \\ \hline
HPLinDer-LP-2  & $55*50$  & 5/7  &  HB-LP & 22    \\ \hline
HPLin-Entr-1  & $111*51$  & 10/15  &  HB-LP-Entropy & 19    \\ \hline
HPLinDer-Entr-1  & $111*51$  & 10/15  &  HB-LP-Entropy & 18    \\ \hline
HPLin-pn-1  & $111*51$  & 10/15  &  HB-LP-pnorm & 30    \\ \hline
HPLinDer-pn-1  & $111*51$  & 10/15  &  HB-LP-pnorm & 32    \\ \hline
Elem-unb-10  & $10*5$  & 10/3  &  HB & 12 (unb)    \\ \hline
Elem-LP-10  & $15*5$  & 10/3  &  HB-LP & 9    \\ \hline
Elem-LP-10-2  & $1010*1000$  & 10/3  &  HB-LP & 30     \\ \hline
Elem-Entr-10-1  & $211*101$  & 10/3  &  HB-LP-Entropy & 17    \\ \hline
Elem-Entr-10-2  & $221*101$  & 10/4  &  HB-LP-Entropy & 18    \\ \hline
\end{tabular}
\end{table}

\subsubsection{Derivatives of determinant of matrix pencils} Let $H_1,\ldots,H_m \in \mathbb H^n$ be Hermitian matrices and $e \in \R^m$ be such that $e_1H_1+\cdots+e_mH_m$ is positive definite. Then, $p(x):= \det(x_1H_1+\cdots+x_mH_m)$ is hyperbolic in direction $e$. Now, by using Theorem \ref{thm:hyper-1}, we can generate a set of interesting hyperbolic polynomials by taking the directional derivative of this polynomial up to $n-3$ times at direction $e$.

\renewcommand{\baselinestretch}{1}
\bibliographystyle{siam}
\bibliography{References}

\appendix

\section{Calculating the predictor and corrector steps} \label{appen:LS}
As discussed in \cite{karimi_arxiv}, for both the predictor and corrector steps, the theoretical algorithm solves a linear system with the LHS matrix of the form
\begin{eqnarray} \label{eq:system-1}
&& U^\top   \left[ \begin{array}{cc} \left [\begin{array} {cc} H & h \\ h^\top & \zeta \end{array} \right] &  0  \\  0  &  \left[ \begin{array}{cc} G+\eta_* h_* h_*^\top  &  -\eta_* h_*  \\  -\eta_* h_*^\top  & \eta_*  \end{array} \right] \end{array} \right] U 
\end{eqnarray}
where $U$ is a matrix that contains the linear transformations we need:
\begin{eqnarray}  \label{eq:mat-U}
U=\left [\begin{array}{ccc}  
A & 0 & 0 \\ 0 & 1 & 0 \\ 0 & -c_A & -F^\top \\ c^\top & 0 & 0 \end{array} \right].
\end{eqnarray}
$H$ and $G$ are positive definite matrices based on the Hessians of the s.c.\ functions, $F$ is a matrix whose rows form a basis for $\Null(A)$, $c_A$ is any vector that satisfies $A^\top c_A=c$, and $\eta_*$, $\zeta$, $h$, and $h_*$ are scalars and vectors defined in \cite{karimi_arxiv}. 
A practical issue of this system is that calculating $F$ is not computationally efficient, and DDS uses a way around it. If we expand the system given in \eqref{eq:system-1}, the linear systems DDS needs to solve become of the form:
\begin{eqnarray*} \label{eq:system-2}
 \left[ \begin{array}{ccc}  A^\top HA +\eta_* cc^\top & A^\top h + \eta_* h_*^\top c_A c  & \eta_* ch_*^\top F^\top  \\
                                         h^\top A+ \eta_* c_A^\top h_* c^\top  &  \zeta+c_A^\top Gc_A+\eta_* (c_A^\top h_*)^2  & c_A^\top GF^\top + \eta_* c_A^\top h_* h_*^\top F^\top \\
                                         \eta_* Fh_* c^\top   &  FGc_A+\eta_* Fh_* h_* ^\top c_A  &  FGF^\top + \eta_* F h_* h_*^\top F^\top \end{array} \right]\left[ \begin{array}{c} \bar d_x \\ d_\tau \\ d_v \end{array}\right]=\left[ \begin{array}{c} r_1 \\ r_2 \\ F r_3 \end{array}\right].
\end{eqnarray*}
At the end, we are interested in $F^\top d_v$ to calculate our search directions. If we consider the last equation, we can remove the matrix $F$ multiplied from the left to all the terms as
\begin{eqnarray} \label{eq:system-3}
&&\eta_* h_* c^\top \bar d_x   +  d_\tau (G+\eta_* h_* h_* ^\top) c_A  +  (G + \eta_* h_* h_*^\top) F^\top d_v  = r_3 + Aw   \nonumber \\
&\Rightarrow& \eta_* \bar G^{-1}  h_* c^\top \bar d_x + d_\tau c_A + F^\top d_v = \bar G^{-1}r_3 + \bar G^{-1}Aw,
\end{eqnarray}
where $\bar G:= G + \eta_* h_* h_*^\top$. Now, we multiply the last equation by $A^\top$ from the left and eliminate $d_v$ to get
\begin{eqnarray} \label{eq:system-4}
 \eta_* A^\top \bar G^{-1}  h_* c^\top \bar d_x + d_\tau  c  = A^\top \bar G^{-1}r_3 + A^\top \bar G^{-1}Aw.
\end{eqnarray}
By using the equations in \eqref{eq:system-3} and \eqref{eq:system-4}, we can get a linear system without $F$ as
\begin{eqnarray} \label{eq:system-5}
\left[ \begin{array}  {ccc} A^\top HA +(\eta_*-\eta_*^2h_*^\top \bar G^{-1}  h_*) cc^\top & \eta_* ch_*^\top \bar G^{-1}A &  A^\top h   \\
\eta_* A^\top \bar G^{-1}  h_* c^\top    &  - A^\top \bar G^{-1}A & c\\
                                  h^\top A   &    c^\top & \zeta   \end{array} \right]\left[ \begin{array}{c} \bar d_x \\ w \\  d_\tau \end{array}\right]=\left[ \begin{array}{c} r_1-  \eta_* ch_*^\top  \bar G^{-1}r_3 \\ A^\top \bar G^{-1}r_3 \\ r_2-c_A^\top r_3 \end{array}\right]. \nonumber \\
\end{eqnarray}
By Sherman–Morrison formula, we can write
\begin{eqnarray} \label{eq:system-6}
&&\bar G^{-1} = G^{-1}- \eta_* \frac{G^{-1} h_* h_*^\top G^{-1}}{1+\eta_* h_*^\top G^{-1} h_*},  \nonumber \\
&\Rightarrow&  \bar G^{-1} h_* = \frac{1}{1+\eta_* h_*^\top G^{-1} h_*} G^{-1} h_*.
\end{eqnarray}
Let us define
\begin{eqnarray} \label{eq:system-8}
\beta:= \frac{1}{1+\eta_* h_*^\top G^{-1} h_*}. 
\end{eqnarray}
Then we can see that the LHS matrix in \eqref{eq:system-5} can be written as 
\begin{eqnarray} \label{eq:system-9}
\underbrace{\left[ \begin{array}  {ccc} A^\top HA & 0 &  A^\top h   \\
  0  &  - A^\top G^{-1}A & c\\
                                  h^\top A   &    c^\top & \zeta   \end{array} \right]}_{\tilde H} + \eta_* \beta \left[ \begin{array}{c} c \\ A^\top G^{-1} h_* \\  0 \end{array}\right]\left[ \begin{array}{c} c \\ A^\top G^{-1} h_* \\  0 \end{array}\right]^\top. 
\end{eqnarray}
This matrix is a $(2n+1)$-by-$(2n+1)$ matrix $\tilde H$ plus a rank one update. If we have the Cholesky or LU factorization of $A^\top HA $ and $A^\top G^{-1}A$ (in the case that $G := \mu^2 H$, we need just one such factorization), then we have such a factorization for the $2n$-by-$2n$ leading minor of $\tilde H$ and we can easily extend it to a factorization for the whole $\tilde H$. To solve the whole system, we can then use Sherman-Morrison formula. 

\section{Implementation details for SDP and generalized epigraphs of matrix norms} \label{sec:imp}
DDS accepts many constraints that involve matrices. These matrices are given into DDS as vectors. Calculating the derivatives and implementing the required linear algebra operations efficiently  is critical in the performance of DDS. In DDS 2.0, many ideas and tricks have been used for matrix functions. In this section, we mention some of these techniques for SDP and  matrix norm constraints. 
 
\subsection{SDP}
The primal and dual barrier functions being used for SDP constraints \eqref{SDP-1} are:
\begin{eqnarray}  \label{SDP-5}
&& \Phi(Z)=-\ln(\det(F_0+Z)), \nonumber  \\
&& \Phi_*(Y)=-n-\langle F_0, Y \rangle - \ln(\det(-Y)).
\end{eqnarray}
For function $f=-\ln(\det(X))$, we have:
\begin{eqnarray} \label{SDP-6}
\langle f'(X), H \rangle &=& -\trace(X^{-1}H), \nonumber \\
\langle f''(X)H,H \rangle &=& \trace(X^{-1}HX^{-1}H).
\end{eqnarray}
To implement our algorithm, for each matrix $X$, we need to find the corresponding gradient $g_X$ and Hessian $H_X$, such that for any symmetric positive semidefinite matrix $X$ and symmetric matrix $H$ we have:
\begin{eqnarray} \label{SDP-7}
-\trace(X^{-1}H) &=& - g_X ^\top \tx{sm2vec}(H), \nonumber \\
\trace(X^{-1}HX^{-1}H) &=& \tx{sm2vec}(H)^ \top H_X \tx{sm2vec}(H).
\end{eqnarray}
It can be shown that $g_X=\tx{sm2vec}(X^{-1})$ and $H_X = X^{-1} \otimes X^{-1}$, where $\otimes$ stands for the Kronecker product of two matrices. Although this representation is theoretically nice, there are two important practical issues with it. First, it is not efficient to calculate the inverse of a matrix explicitly. Second, forming and storing  $H_X$ is not numerically efficient  for large scale problems. DDS does not explicitly form the inverse of matrices. An important matrix in calculating the search direction is  $A^\top \Phi''(u)  A$, where $\Phi$ is the s.c.\ barrier for the whole input problem.  In DDS, there exist an internal function \tx{hessian\_A}
that directly returns $A^\top \Phi''(u)  A$ in an efficient way, optimized for all the set constraints, including SDP.  
For transitions from matrices to vectors, we use the properties of Kronecker product  that for matrices $A$, $B$, and $X$ of proper size, we have
\begin{eqnarray} \label{SDP-9}
&& (B^\top \otimes A) sm2vec(X) = sm2vec(AXB), \nonumber \\
&& (A \otimes B) ^{-1} = A^{-1} \otimes B^{-1}.
\end{eqnarray}
Similar to \tx{hessian\_A}, there are other internal functions in DDS for efficiently calculating matrix-vector products, such as \tx{hessian\_V} for evaluating the product of Hessian with a given vector of proper size.

\subsection{Generalized epigraphs of matrix norms} \label{sec:imp:mn}

Let us see how to calculate the first and second derivatives of functions in \eqref{EO2N-3} and \eqref{EO2N-4}. 
\begin{proposition}
(a) Consider $\Phi(X,U)$ defined in \eqref{EO2N-3}. Let, for simplicity, $\bar X:=X-UU^\top$, then, we have
\begin{eqnarray} \label{EO2N-9}
\Phi'(X,U)[(d_X,d_U)] &=& \trace(-\bar X^{-1}d_X+\bar X^{-1} (d_UU^\top + U d_U ^\top)),  \nonumber \\
\Phi''(X,U)[(d_X,d_U),(\bar d_X, \bar d_U)] &=&   \trace(\bar X^{-1}d_X\bar X^{-1} \bar d_X)  \nonumber \\
                                   &&  -  \trace(\bar X^{-1} \bar d_X\bar X^{-1} (d_UU^\top + U d_U ^\top))  \nonumber \\ 
                                   && -\trace(\bar X^{-1}d_X\bar X^{-1} (\bar d_UU^\top + U \bar d_U ^\top)) \nonumber \\
                                   && + \trace(\bar X^{-1}(d_UU^\top + U d_U ^\top)\bar X^{-1} (\bar d_UU^\top + U \bar d_U ^\top)) \nonumber \\ 
                                   && + 2 \trace(\bar X^{-1}d_U \bar d_U^\top).
\end{eqnarray}
(b) Consider $\Phi_*(Y,V)$ defined in \eqref{EO2N-4}, we have
\begin{eqnarray*} \label{EO2N-10}
\Phi'_*(Y,V)[(d_Y,d_V)] &=& -\frac{1}{2} \trace(V^\top Y^{-1} d_V)+\frac{1}{4}\trace(Y^{-1} VV^\top Y^{-1} d_Y)-\trace(Y^{-1} d_Y),  \nonumber \\
\Phi''_*(Y,V)[(d_Y,d_V),( \bar d_Y, \bar d_V)] &=&  -\frac{1}{2} \trace(d_V^\top Y^{-1} \bar d_V)  \nonumber \\
                                   &&  + \frac{1}{2} \trace(Y^{-1} d_VV^\top Y^{-1} \bar d_Y)  + \frac{1}{2} \trace(Y^{-1} \bar d_VV^\top Y^{-1} d_Y) \nonumber \\
                                   && -\frac{1}{2} \trace(Y^{-1} d_Y Y^{-1} \bar d_Y Y^{-1} VV^\top)+  \trace(Y^{-1} d_Y Y^{-1} \bar d_Y).
\end{eqnarray*}
\end{proposition}
\begin{proof}
For the proof we use the fact that if $g=-\ln(\det(X))$, then $g'(X)[H]=\trace(X^{-1}H)$. Also note that if we define
\begin{eqnarray} \label{EO2N-11}
g(\alpha):=-\ln(\det(X+\alpha d_X-(U+\alpha d_U)(U+\alpha d_U)^\top)),
\end{eqnarray}
then
\[
g'(0)=\Phi'(X,U)[(d_X,d_U)], \ \ \ g''(0)=\Phi''(X,U)[(d_X,d_U),(d_X,d_U)],
\]
and similarly for $\Phi_*(Y,V)$. We do not provide all the details, but we show how the proof works. For example, let us define
\begin{eqnarray} \label{EO2N-12}
f(\alpha):=\trace((Y+\alpha d_Y)^{-1} VV^\top Y^{-1} d_Y),
\end{eqnarray}
and we want to calculate $f'(0)$. We have
\begin{eqnarray} \label{EO2N-13}
f'(0)&:=&\lim_{\alpha \rightarrow 0} \frac{f(\alpha)-f(0)}{\alpha}  \nonumber \\
      &=& \trace \left ( \lim_{\alpha \rightarrow 0}  \frac{(Y+\alpha d_Y)^{-1} VV^\top Y^{-1} d_Y-Y^{-1} VV^\top Y^{-1} d_Y}{\alpha}\right) \nonumber \\
      &=& \trace \left ( \lim_{\alpha \rightarrow 0}  \frac{(Y+\alpha d_Y)^{-1} \left [ VV^\top Y^{-1} d_Y-(I+\alpha d_Y Y^{-1}) VV^\top Y^{-1} d_Y  \right]}{\alpha}\right)\nonumber \\
      &=& \trace \left ( \lim_{\alpha \rightarrow 0}  (Y+\alpha d_Y)^{-1} \left [ d_Y Y^{-1}VV^\top Y^{-1} d_Y  \right]\right)\nonumber \\
      &=& \trace \left ( Y^{-1} d_Y Y^{-1}VV^\top Y^{-1} d_Y  \right). 
\end{eqnarray}
\end{proof}
Note that all the above formulas for the derivatives are in matrix form. Let us explain briefly how to convert them to the vector form for the code. We explain it for the derivatives of $\Phi(X,U)$ and the rest are similar. From \eqref{EO2N-9} we have
\begin{eqnarray} \label{EO2N-14}
\Phi'(X,U)[(d_X,d_U)] &=& \trace(-\bar X^{-1}d_X)+\trace(\bar X^{-1} d_UU^\top) + \trace(X^{-1}U d_U ^\top)),  \nonumber \\
&=& \trace(-\bar X^{-1}d_X)+2\trace(U^\top \bar X^{-1} d_U).
\end{eqnarray}
Hence, if $g$ is the gradient of $\Phi(X,U)$ in the vector form, we have
\begin{eqnarray} \label{EO2N-15}
g= \left [ \begin{array} {c}  2 \times m2vec(X^{-1}U ,n)  \\ -sm2vec(X^{-1}) \end{array} \right].
\end{eqnarray}
The second derivatives are trickier. Assume that for example we want the vector form $h$ for $\Phi''(X,U)[(d_X,d_U)]$. By using \eqref{EO2N-9} we can easily get each entry of $h$; consider the identity matrix of size $m^2+mn$. If we choose $(\bar d_X, \bar d_U)$ to represent the $j$th column of this identity matrix, we get $h(j)$. Practically, this can be done by a $for$ loop, which is not efficient. What we did in the code is to implement this using matrix multiplication. 

\section{Quantum entropy and quantum relative entropy} \label{appen:quantum}
The s.c.\ barrier DDS uses for quantum entropy is given in \eqref{eq:fun_cal_3}. To derive the LF conjugate, we solve the optimization problem in  \eqref{eq:fun_cal_4}; let us write the first order optimality conditions for \eqref{eq:fun_cal_4}. Section 3.3 of the book \cite{hiai2014introduction} is about the derivation of matrix-values functions. For the first derivative, we have the following theorem:
\begin{theorem} \label{thm:mtx_fun_der_1}
Let $X$ and $H$ be self-adjoint matrices and $f: (a,b)\mapsto \mathbb R$ be a continuously differentiable function defined on an interval. Assume that the eigenvalues of $X+\alpha H$ are in $(a,b)$ for an interval around $\alpha_0 \in \mathbb R$. Then, 
\begin{eqnarray}\label{eq:fun_cal_5}
\left. \frac{d}{dt}\trace{f(X+\alpha H)}\right|_{\alpha=\alpha_0} = \trace{Hf'(X+\alpha_0H)}. 
\end{eqnarray}
\end{theorem}
The first-order optimality conditions for \eqref{eq:fun_cal_4} can be written as
\begin{eqnarray}\label{eq:fun_cal_6}
\eta+\frac{1}{t-\phi(X)} &=&0 \nonumber \\
Y+\frac{-f'(X)}{t-\phi(X)} +X^{-1}&=&0. 
\end{eqnarray}
If we substitute the first equation in the second one, we get
\begin{eqnarray}\label{eq:fun_cal_7}
\frac{1}{\eta}Y+f'(X) +\frac{1}{\eta}X^{-1}=0. 
\end{eqnarray}
This equation implies that $Y$ and $X$ are simultaneously diagonalizable and if we have $Y=U\Diag(\lambda_1(Y),\ldots,\lambda_n(Y))$, then we have $X=U\Diag(\lambda_1(X),\ldots,\lambda_n(X))$ and so
\begin{eqnarray}\label{eq:fun_cal_8}
\frac{1}{\eta}\lambda_i(Y)+f'(\lambda_i(X)) +\frac{1}{\eta \lambda_i(X)}=0, \ \ \ i\in\{1,\ldots,n\}. 
\end{eqnarray}
Here, we focus on the case that $f(x)=x\ln(x)$. This matrix function is related to quantum relative entropy and Von-Neumann entropy optimization problems (see \cite{chandrasekaran2017relative} for a review of the applications). In this case, we can use results for type 3 univariate function in Table \ref{table1} and use the $\theta$ function we defined in \eqref{eq:entropy_theta_1}. The LF conjugate of \eqref{eq:fun_cal_4} is given in the following lemma:
\begin{lemma}
Assume that $f(x)=x\ln(x)$. For a given $\eta < 0$ and a symmetric matrix $Y \in \mathbb S^n$, the function defined in \eqref{eq:fun_cal_4} becomes
\begin{eqnarray}\label{eq:fun_cal_9}
\Phi_*(\eta,Y)= -\ln(-\eta) + \trace(\Theta+\Theta^{-1}) - \trace\left(\frac{1}{\eta}Y\right) -1-2n,
\end{eqnarray}
where $\Theta:= \Theta(\frac{1}{\eta}Y+(1-\ln(-\eta))I)$. $\Theta$ is the matrix extension of $\theta$ defined in \eqref{eq:entropy_theta_1} as described in Section \ref{sec:QE}. 
\end{lemma}
\begin{proof}
Assume that for a given $(\eta,Y)$, $(t,X)$ is the optimal solution for \eqref{eq:fun_cal_4}. If we use theorem \ref{thm:mtx_fun_der_1}, we have $f'(X)=I+\ln(X)$. By substituting this in the first order optimality condition \eqref{eq:fun_cal_7} we get
\begin{eqnarray}\label{eq:fun_cal_10}
\eta X= -\Theta \left (\frac{1}{\eta}Y+(1-\ln(-\eta))I\right). 
\end{eqnarray}
By the first equation in \eqref{eq:fun_cal_6} and using \eqref{eq:fun_cal_10} we can write
\begin{eqnarray}\label{eq:fun_cal_11}
\eta t = -1+\trace(\eta X \ln(X))&=& -1+\trace (-\Theta \cdot \ln(X)) \nonumber \\
&=& -1+\trace \left(\Theta \cdot \left(\frac{1}{\eta}Y+I-\Theta^{-1}\right) \right). \nonumber \\
    &=&-1-n+\trace(\Theta Y/\eta) + \trace(\Theta). 
\end{eqnarray}
If we substitute $t$ and $X$ in  \eqref{eq:fun_cal_4}, we get the result. 
\end{proof}
To implement our primal-dual techniques, we need the gradient and Hessian of $\Phi(t,X)$ and $\Phi_*(\eta,Y)$. We already saw in Theorem \ref{thm:mtx_fun_der_1} how to calculate the gradient. The following theorem gives us a tool to calculate the Hessian. 
\begin{theorem} [\cite{hiai2014introduction}-Theorem 3.25] \label{thm:mtx_fun_der_2}
Assume that $f:(a,b) \mapsto \R$ is a $\mathcal C^1$-function and $T=\Diag(t_1,\ldots,t_n)$ with $t_i \in (a,b), \ i\in\{1,\ldots,n\}$. Then, for a Hermitian matrix $H$, we have
\begin{eqnarray} \label{eq:thm:mtx_fun_der_2_1}
\left. \frac{d}{dt} f(T+\alpha H) \right|_{\alpha=0}=T_f \odot H, 
\end{eqnarray}
where $\odot$ is the Hadamard product and $T_f$ is the divided difference matrix:
\begin{eqnarray} \label{eq:thm:mtx_fun_der_2_2}
T_f:= \left\{\begin{array}{ll} \frac{f(t_i)-f(t_j)}{t_i-t_j} & t_i \neq t_j \\ f'(t_i) & t_i=t_j\end{array} \right..
\end{eqnarray}
\end{theorem}
$T$ is diagonal in the statement of the theorem, which is without loss of generality. Note that by the definition of functional calculus in \eqref{eq:fun_cal_1}, for a Hermitian matrix $X$ and a unitary matrix $U$, we have
\begin{eqnarray}\label{eq:fun_cal_13}
f(UXU^*)=Uf(X)U^*. 
\end{eqnarray}
Therefore,  for a matrix $T=U\Diag(t_1,\ldots,t_n)U^*$, we can update \eqref{eq:thm:mtx_fun_der_2_1}
\begin{eqnarray} \label{eq:thm:mtx_fun_der_2_3}
\left. \frac{d}{dt} f(T+\alpha H) \right|_{\alpha=0}=U \left(T_f \odot (U^*HU)\right) U^*,
\end{eqnarray}
where we extend the definition of $T_f$ in \eqref{eq:thm:mtx_fun_der_2_2} to non-diagonal matrices. 
Now we can use Theorems  \ref{thm:mtx_fun_der_2} and \ref{thm:mtx_fun_der_1} to calculate the Hessian of a matrix function.
\begin{corollary}
Let $X$, $H$, and $\tilde H$ be self-adjoint matrices and $f: (a,b)\mapsto \mathbb R$ be a continuously differentiable function defined on an interval. Assume that the eigenvalues of $X+tH$ and $X+t \tilde H$ are in $(a,b)$ for an interval around $t=0$. Assume that $X=U\Diag(\lambda_1,\ldots,\lambda_n)U^*$. Then, 
\begin{eqnarray}\label{eq:fun_cal_14}
f''(X)[H,\tilde H]=\trace \left( \left(X_f \odot (U^*HU)\right) U^* \tilde HU \right). 
\end{eqnarray}
\end{corollary}
Let us calculate the gradient and Hessian for our functions for $\phi(x)=x\ln(x)$. Let $X=U\Diag(\lambda_1,\ldots,\lambda_n)U^*$ in the following. 
\begin{eqnarray*}\label{eq:fun_cal_15}
\Phi'(t,X)[(h,H)]= -\frac{h}{t-\trace (X\ln X)} + \frac{1}{t-\trace (X\ln X)} \trace((I+\ln(X))H)-\trace(X^{-1}H).
\end{eqnarray*}
For the second derivative, we can use the fact that
\begin{eqnarray*}\label{eq:fun_cal_16}
\Phi''(t,X)[(\tilde h, \tilde H),(h,H)]= \left. \Phi'(t+\alpha \tilde h,X+\alpha \tilde H)\right|_{\alpha=0}[(h,H)].
\end{eqnarray*}
Using this formula, we have ($\zeta:= \frac{ 1}{t-\trace (X\ln X)}$)
\begin{eqnarray*}\label{eq:fun_cal_17}
\Phi''(t,X)[(\tilde h, \tilde H),(h,H)]&=& \zeta^2 h\tilde h\nonumber \\
&& - \zeta^2 \tilde h \trace((I+\ln(X))H)- \zeta h\trace((I+\ln(X))\tilde H) \nonumber \\
&& +\zeta^2 \trace((I+\ln(X))\tilde H)\trace((I+\ln(X))H) \nonumber \\
&& +\zeta \trace\left(U \left(X_{\ln} \odot (U^*\tilde HU)\right) U^*H\right) \nonumber \\
&&+ \trace(X^{-1}\tilde HX^{-1}H).
\end{eqnarray*}
Now let us compute the gradient and Hessian for the conjugate function. Let $Y=U\lambda(Y)U^*$, by using Theorem \ref{thm:mtx_fun_der_1}, the gradient of $\Phi_*(\eta,Y)$ is
\begin{eqnarray*}\label{eq:fun_cal_12}
\Phi'_*(\eta,Y)[(h,H)]&=&h \left[-\frac{1}{\eta}+\trace\left(\left(-\frac{1}{\eta^2} Y-\frac{1}{\eta}I\right) \bar \Theta+\frac{1}{\eta^2} Y\right) \right] \nonumber \\
                   &&+ \trace\left( H \left(\frac{1}{\eta}\bar \Theta -\frac{1}{\eta}I\right) \right),
\end{eqnarray*}
where $\bar \Theta$ is the matrix extension of $\left(\theta'-\frac{\theta'}{\theta^2}\right)$. 
For the second derivative, let us first define $\bar Y$ as in \eqref{eq:thm:mtx_fun_der_2_2}:
\begin{eqnarray*}\label{eq:fun_cal_19}
\bar Y:= \left(\frac{1}{\eta}Y+(1-\ln(-\eta))I\right)_{\left(\theta'-\frac{\theta'}{\theta^2}\right)}.
\end{eqnarray*}
In the expression of the Hessian, we use $\bar{\bar \Theta}$ as the matrix extension of $\left(\theta''-\frac{\theta''\theta-2(\theta')^2}{\theta^3}\right)$.  Then, we have
\begin{eqnarray*}\label{eq:fun_cal_18}
&&\Phi''_*(\eta,Y)[(\tilde h,\tilde H),(h,H)]= \nonumber \\
&& h\tilde h \left[\frac{1}{\eta^2}+\trace\left(\left(\frac{2}{\eta^3} Y+\frac{1}{\eta^2}I\right)\bar \Theta+\left(-\frac{1}{\eta^2} Y-\frac{1}{\eta}I\right)^2 \bar{\bar \Theta} -\frac{2}{\eta^3} Y\right) \right] \nonumber \\
                   &&+ \tilde h\trace\left( H \left[\frac{-1}{\eta^2} \bar \Theta+ \frac{1}{\eta}U\left(\bar Y \odot \left(\frac{-1}{\eta^2}\lambda(Y)-\frac{1}{\eta}I\right)\right)U^* +\frac{1}{\eta^2} I\right] \right) \nonumber \\
                   &&+  h\trace\left( \tilde H \left[\frac{-1}{\eta^2} \bar \Theta+ \frac{1}{\eta}U\left(\bar Y \odot \left(\frac{-1}{\eta^2}\lambda(Y)-\frac{1}{\eta}I\right)\right)U^* +\frac{1}{\eta^2} I\right] \right) \nonumber \\
                   &&+ \frac{1}{\eta^2} \trace\left(U \left(\bar Y \odot (U^*\tilde HU)\right) U^*H\right).
\end{eqnarray*}

\section{Univariate convex functions} \label{app:LF-univar}

DDS 2.0 uses the epigraph of six univariate convex function, given in  Table \ref{table1}, in the direct sum format to handle important constraints such as geometric and entropy programming. In these section, we show how to calculate the LF conjugate for 2-variable s.c.\ barriers. Some of these function are implicitly evaluated, but in a computationally efficient way. In the second part of this section, we show that gradient and Hessian of these functions can also be evaluated efficiently. 

\subsection{Legendre-Fenchel conjugates of univariate convex functions}
The LF conjugates for the first three functions are given in Table \ref{table:2dim-1}.
\begin{table}  [h]
\caption{LF conjugates for the first three s.c.\ barriers in Table \ref{table1}. }
\label{table:2dim-1}
\center
\renewcommand*{\arraystretch}{1.3}
  \begin{tabular}{ | c | c | c| }
    \hline
      & { $\Phi(z,t)$}  & $\Phi_*(y,\eta)$ \\  \hline
    1 & $-\ln(t+\ln(z))-\ln(z)$  & $-1+(-\eta+1) \left [ -1+\ln\frac{-(-\eta+1)}{y}\right ]-\ln(-\eta)$ \\ \hline
    2  & $-\ln(\ln(t)-z)-\ln(t)$ & $-1+(y+1)\left [ -1+\ln \frac{-(y+1)}{\eta}\right] - \ln(y)$ \\ \hline
    3 &  $-\ln(t-z\ln(z)) - \ln(z)$ &  $-\ln(-\eta) + \theta\left( 1+ \frac{y}{\eta} - \ln(-\eta)\right) - \frac{y}{\eta} + \frac{1}{ \theta\left( 1+ \frac{y}{\eta} - \ln(-\eta)\right)} -3$ \\ \hline
  \end{tabular}
\end{table}
Finding the LF conjugates for the first two functions can be handled with easy calculus. In the third row, $\theta(r)$, defined in \cite{cone-free}, is the unique solution of
\begin{eqnarray} \label{eq:entropy_theta_1}
\frac{1}{\theta} - \ln(\theta) = r.
\end{eqnarray}
It is easy to check by implicit differentiation that
\begin{eqnarray*}
\theta'(r) = - \frac{\theta^2(r)}{\theta(r) +1}, \ \ \theta''(r) = \frac{\theta^2(r)+2\theta(r)}{[\theta(r)+1]^2} \theta'(r).
\end{eqnarray*}
We can calculate $\theta(r)$ with accuracy $10^{-15}$ in few steps with the following Newton iterations:
\begin{eqnarray*}
\theta_k=\frac{\theta^2_{k-1}}{\theta_{k-1}+1}\left [ 1+\frac{2}{\theta_{k-1}}-\ln(\theta_{k-1}) - r\right], \ \ \theta_0 = \left \{ \begin{array} {ll}  \exp(-r),  & r \leq 1 \\ \frac{1}{r-\ln(r-\ln(r))}, & r>1\end{array}\right..
\end{eqnarray*} 

The s.c.\ barrier DDS uses for the set  $|z|^p \leq t, \ p \geq 1$, is $\Phi(z,t)=-\ln(t^{\frac 2p} - z^2) - 2\ln(t)$. To find the LF conjugate, we need to solve the following optimization problem:
\begin{eqnarray}
\min_{z,t} \left \{ yz+\eta t+  \ln(t^{\frac 2p} - z^2) + 2\ln(t) \right \}.
\end{eqnarray}
By writing the first order optimality conditions we have:
\begin{eqnarray}
y=\frac{2z}{t^{\frac 2p} - z^2}, \ \ \ \eta=-\frac{\frac{2}{p} t^{\frac 2p -1}}{t^{\frac 2p} - z^2} - \frac 2t.
\end{eqnarray}
By doing some algebra, we can see that $z$ and $t$ satisfy:
\begin{eqnarray} \label{2-dim-1}
&& y\left ( \frac{2(\frac 1p+1)+ \frac 1p yz}{-\eta} \right )^{\frac 2p} - yz^2-2z =0,  \nonumber \\
&& t=  \frac{2(\frac 1p+1)+ \frac 1p yz}{-\eta}.
\end{eqnarray}
Let us define $z(y,\eta)$ as the solution of the first equation in \eqref{2-dim-1}. For each pair $(y,\eta)$, we can calculate $z(y,\eta)$ by few iterations of Newton method. Then, the first and second derivative 
can be calculated in terms of $z(y,\eta)$.  
 In DDS, we  have two functions for these derivatives. 
\begin{verbatim}
p1_TD(y,eta,p)    % returns  z
p1_TD_der(y,eta,p)  % returns [z_y  z_eta  z_y,y  z_y,eta  z_eta,eta]
\end{verbatim}

For the set defined by  $-z^p \leq t, \ 0 \leq p \leq 1, z > 0$, the corresponding  s.c.\ barrier is $\Phi(z,t)=-\ln(z^p+t)-\ln(z)$. To calculate the 
LF conjugate, we need to solve the following optimization problem:
\begin{eqnarray}
\min_{z,t} \left \{ yz+\eta t+  \ln(z^p+t)+\ln(z) \right \}.
\end{eqnarray}
By writing the first order optimality conditions we have:
\begin{eqnarray}
y=\frac{-pz^{(p-1)}}{z^p+t} - \frac 1z, \ \ \ \eta=-\frac{1}{z^p+t}.
\end{eqnarray}
By doing some algebra, we can see that $z$ satisfies:
\begin{eqnarray} \label{2-dim-1-2}
 y-\eta p z^{(p-1)}+\frac 1z=0.
\end{eqnarray}

Similar to the previous case, let us define $z(y,\eta)$ as the solution of the first equation in \eqref{2-dim-1-2}. For each pair $(y,\eta)$, we can calculate $z(y,\eta)$ by few iterations of Newton method. Then, the first and second derivatives
can be calculated in terms of $z(y,\eta)$. The important point is that when we calculate $z(y,\eta)$, then the derivatives can be calculated by explicit formulas. In our code, we 
have two functions
\begin{verbatim}
p2_TD(y,eta,p)    % returns  z
p2_TD_der(y,eta,p)  % returns [z_y  z_eta  z_y,y  z_y,eta  z_eta,eta]
\end{verbatim}
The inputs to the above functions can be vectors. Table \ref{table1-2} is the continuation of Table \ref{table:2dim-1}. 
\begin{table} [h] 
  \caption{s.c.\ barriers and their LF conjugate for rows 4 and 5 of Table \ref{table1}}
  \label{table1-2}
  \center
  \renewcommand*{\arraystretch}{1.3}
  \begin{tabular}{ |c | c | c | }
    \hline
     &  {s.c.\ barrier $\Phi(z,t)$}  & $\Phi_*(y,\eta)$ \\  \hline
     4 &  $-\ln(t^{\frac 2p} - z^2) - 2\ln(t)$  & $- \left( \frac 2p + (\frac 1p -1) yz(y,\eta)\right) - 2 + 2 \ln \left (  \frac{2(\frac 1p+1)+ \frac 1p yz(y,\eta)}{-\eta} \right )$ \\ 
        &  &  $+\ln \left ( \left ( \frac{2(\frac 1p+1)+ \frac 1p yz(y,\eta)}{-\eta} \right )^{\frac 2p} - z^2(y,\eta)\right )$  \\  \hline
    5  & $-\ln(z^p+t)-\ln(z)$ &  $\eta (p-1) z^p(y,\eta) - 2 - \ln(-\eta) + \ln(z(y,\eta))$ \\ \hline 
  \end{tabular}
\end{table}

For the set defined by $\frac{1}{z} \leq t, z > 0$, the corresponding  s.c.\ barrier is $\Phi(z,t)=-\ln(zt-1)$. To calculate the LF conjugate, we need to solve the following optimization problem:
\begin{eqnarray}
\min_{z,t} \left \{ yz+\eta t+ \ln(zt-1) \right \}.
\end{eqnarray}
At the optimal solution, we must have
\begin{eqnarray}
y=-\frac{t}{zt-1}, \ \ \eta=-\frac{z}{zt-1}. 
\end{eqnarray}
Since we have $z,t >0$, then we must have $y,\eta <0$ at a dual feasible point. By solving these systems we get
\begin{eqnarray}
&& t=\frac{-1-\sqrt{1+4y \eta}}{2\eta},  \ \ z=\frac{-1-\sqrt{1+4y \eta}}{2y},  \nonumber \\
&\Rightarrow& \Phi_*(y,\eta)=-1-\sqrt{1+4y \eta}+\ln\left(\frac{1+\sqrt{1+4y \eta}}{2y \eta}\right).
\end{eqnarray}

\subsection{Gradient and Hessian}  \label{formulas}
In this subsection, we show how to calculate the gradient and Hessian for the s.c.\ barriers and their LF conjugates. Specifically for the implicit function, we show that the derivatives can also be calculated efficiently. 
First consider the three pairs of functions in Table \ref{table:2dim-1}. Here are the explicit formulas for the first and second derivatives: 

  \begin{tabular}{  | c | c| }
    \hline
      { $\Phi(z,t)$}  & $\Phi_*(y,\eta)$ \\  \hline
     $-\ln(t+\ln(z))-\ln(z)$  & $-1+(-\eta+1) \left [ -1+\ln\frac{-(-\eta+1)}{y}\right ]-\ln(-\eta)$ \\ \hline
  \end{tabular}
  
  For the primal function we have
\begin{eqnarray*}  \label{eqn:app-5}
\nabla \Phi =  \left [ \begin{array} {c}  -\frac{1}{z} \left( \frac{1}{t+\ln(z)}+1\right)  \\  -\frac{1}{t+\ln(z)} \end{array}\right],
 \ \ \nabla^2 \Phi =  \left [ \begin{array} {cc} \frac{1}{z^2} \left( \frac{1}{t+\ln(z)} + \frac{1}{(t+\ln(z))^2}+1\right)  &  
  \frac{1}{z(t+\ln(z))^2} \\ \frac{1}{z(t+\ln(z))^2}  & \frac{1}{(t+\ln(z))^2}  \end{array}\right],
\end{eqnarray*}
and for the dual function we have
\begin{eqnarray*}  \label{eqn:app-6}
\nabla \Phi_* =  \left [ \begin{array} {c}  -\frac{-\eta+1}{y}  \\  -\ln \left( -\frac{-\eta+1}{y}\right) -\frac{1}{\eta}  \end{array}\right],
 \ \ \nabla^2 \Phi_* =  \left [ \begin{array} {cc} \frac{-\eta+1}{y^2}  &  \frac{1}{y}  
  \\ \frac{1}{y}  & \frac{1}{-\eta+1}+\frac{1}{\eta^2}  \end{array}\right].
\end{eqnarray*}

\begin{tabular}{  | c | c| }
    \hline
      { $\Phi(z,t)$}  & $\Phi_*(y,\eta)$ \\  \hline
     $-\ln(\ln(t)-z)-\ln(t)$ & $-1+(y+1)\left [ -1+\ln \frac{-(y+1)}{\eta}\right] - \ln(y)$ \\ \hline
  \end{tabular}
  
  For the primal function we have
\begin{eqnarray*}  \label{eqn:app-5-2}
\nabla \Phi =  \left [ \begin{array} {c}   \frac{1}{\ln(t)-z}  \\  \frac{1}{t} \left ( \frac{1}{\ln(t)-z} +1 \right) \end{array}\right],
 \ \ \nabla^2 \Phi =  \frac{1}{(\ln(t)-z)^2}   \left [ \begin{array} {cc} 1  &  
  -\frac{1}{t} \\ -\frac{1}{t}  & \frac{1+(\ln(t)-z)+(\ln(t)-z)^2}{t^2} \end{array}\right],
\end{eqnarray*}
and for the dual function we have
\begin{eqnarray*}  \label{eqn:app-6-2}
\nabla \Phi_* =  \left [ \begin{array} {c}  \ln \left (\frac{-(y+1)}{\eta}\right)-\frac{1}{y}  \\  -\frac{y+1}{\eta}  \end{array}\right],
 \ \ \nabla^2 \Phi_* =  \left [ \begin{array} {cc} \frac{1}{y+1}+\frac{1}{y^2}  &  -\frac{1}{\eta}  
  \\ -\frac{1}{\eta}  & \frac{y+1}{\eta^2}  \end{array}\right].
\end{eqnarray*}

\begin{tabular}{  | c | c| }
    \hline
      { $\Phi(z,t)$}  & $\Phi_*(y,\eta)$ \\  \hline
     $-\ln(t-z\ln(z)) - \ln(z)$ &  $-\ln(-\eta) + \theta\left( 1+ \frac{y}{\eta} - \ln(-\eta)\right) - \frac{y}{\eta} + \frac{1}{ \theta\left( 1+ \frac{y}{\eta} - \ln(-\eta)\right)} -3$ \\ \hline
  \end{tabular}
  
  For the primal function we have
\begin{eqnarray*}  \label{eqn:app-5-3}
\nabla \Phi =  \left [ \begin{array} {c}   \frac{\ln(z)+1}{t-z\ln(z)}-\frac{1}{z}  \\  \frac{-1}{t-z\ln(z)} \end{array}\right],
 \ \ \nabla^2 \Phi =     \left [ \begin{array} {cc} \frac{(t-z\ln(z))+(\ln(z)+1)^2}{z(t-z\ln(z))^2}+\frac{1}{z^2}  &  \frac{-(\ln(z)+1)}{(t-z\ln(z))^2} \\
  \frac{-(\ln(z)+1)}{(t-z\ln(z))^2}  & \frac{1}{(t-z\ln(z))^2} \end{array}\right].
\end{eqnarray*}
For the dual function, since the argument of the function $\theta(\cdot)$ is always $1+ \frac{y}{\eta} - \ln(-\eta)$, we ignore that in the following formulas and use $\theta$, $\theta'$, and $\theta''$ for the function and its derivative. 
\begin{eqnarray*}  \label{eqn:app-6-3}
\nabla \Phi_* =  \left [ \begin{array} {c}  \frac{\theta'-1}{\eta}-\frac{\theta'}{\eta \theta^2}  \\  -\frac{1}{\eta}+\frac{y}{\eta^2}-\left ( \frac{y}{\eta^2}+\frac{1}{\eta} \right)\theta' \left ( 1+\frac{1}{\theta^2}\right)  \end{array}\right],
 \ \ \nabla^2 \Phi_* =  \left [ \begin{array} {cc} f_{11}  &  f_{12} 
  \\ f_{21}  & f_{22}  \end{array}\right],
\end{eqnarray*}
where
\begin{eqnarray*}  \label{eqn:app-6-3-2}
f_{11} &=& \frac{1}{\eta^2}\theta''-\frac{\theta''\theta-2(\theta')^2}{\eta^2 \theta^3}, \nonumber \\
f_{21}=f_{12}&=& -\frac{1}{\eta^2}\theta' + \frac{1}{\eta} \left ( -\frac{y}{\eta^2}-\frac{1}{\eta} \right) \theta'' + \frac{1}{\eta^2}- \frac{\left[-\frac{1}{\eta^2}\theta' + \frac{1}{\eta} \left ( -\frac{y}{\eta^2}-\frac{1}{\eta} \right) \theta''\right] \theta - \frac{2}{\eta} \left ( -\frac{y}{\eta^2}-\frac{1}{\eta} \right) (\theta')^2}{\theta^3} \nonumber \\
f_{22} &=& \frac{1}{\eta^2}-\frac{2y}{\eta^3}+\left[ \left ( \frac{2y}{\eta^3}+\frac{1}{\eta^2}\right)\theta' + \left ( -\frac{y}{\eta^2}-\frac{1}{\eta}  \right)^2 \theta'' \right]\left ( 1+\frac{1}{\theta^2}\right)
    +\left ( -\frac{y}{\eta^2}-\frac{1}{\eta}  \right)^2 \frac{2(\theta')^2}{\theta^3}
\end{eqnarray*}

\begin{tabular}{  | c | c| }
    \hline
      { $\Phi(z,t)$}  & $\Phi_*(y,\eta)$ \\  \hline
     $-\ln(t^{\frac 2p} - z^2) - 2\ln(t)$  & $- \left( \frac 2p + (\frac 1p -1) yz\right) - 2 + 2 \ln \left (  \frac{2(\frac 1p+1)+ \frac 1p yz}{-\eta} \right )+\ln \left ( \left ( \frac{2(\frac 1p+1)+ \frac 1p yz}{-\eta} \right )^{\frac 2p} - z^2\right )$ \\ \hline
  \end{tabular}
  
where $z(y,\eta)$ is the solution of 
\begin{eqnarray} \label{eqn:app-7-1}
y\left ( \frac{2(\frac 1p+1)+ \frac 1p yz}{-\eta} \right )^{\frac 2p} - yz^2-2z =0.
\end{eqnarray}
For simplicity, we drop the arguments of $z(y,\eta)$ and denote it as $z$. We denote the first derivatives with respect to $y$ and $\eta$ as $z'_y$ and $z'_\eta$, respectively. Similarly, we use $z''_{yy}$, $z''_{\eta y}$, and $z''_{\eta\eta}$ for the second derivatives.  We have
\begin{eqnarray}
z'_y &=& \frac{B^{\frac 2p} - \frac{2y}{p^2\eta} x B^{\frac 2p-1} -  z^2   \ \ \ =: S }
      {\frac{2y^2}{p^2\eta}B^{\frac 2p-1} + 2yz +2  \ \ \ =: M }   \nonumber \\
z'_{\eta} &=&      \frac {\frac{-2y}{p\eta} B^{\frac 2p}   \ \ \ =: T} {\frac{2y^2}{p^2\eta}B^{\frac 2p-1} + 2yz +2 }  \nonumber \\
B &:=&  \frac{2(\frac 1p+1)+ \frac 1p yz}{-\eta} 
\end{eqnarray}
For calculating the second derivatives of $\Phi_*$, we need the derivatives of $B$:
\begin{eqnarray}
B'_y&=&\frac{z+yz'_y}{-p\eta}, \nonumber \\
B'_\eta &=& \frac{-\frac{\eta}{p}yz'_\eta+2 (\frac 1p+1) + \frac 1p yz}{\eta^2}.
\end{eqnarray}
Then we have
\begin{eqnarray*}
S'_y &=& \frac{2}{p}B'_yB^{\frac 2p-1}+\left( \frac{-2z}{p^2 \eta} - \frac{2y}{p^2 \eta} z'_y \right) B^{\frac 2p-1} + \frac{-2yz}{p^2\eta} \left(\frac 2p-1 \right) B'_y B^{\frac 2p-2}-2zz'_y, \nonumber \\
S'_\eta &=& \frac{2}{p}B'_\eta B^{\frac 2p-1}+\frac{-2y}{p^2} \left ( \frac{\eta z'_\eta - z}{\eta^2} \right)B^{\frac 2p-1}+\frac{-2yz}{p^2\eta} \left(\frac 2p-1 \right) B'_\eta B^{\frac 2p-2}-2zz'_\eta, \nonumber \\
\end{eqnarray*}
\begin{eqnarray*}
M'_y &=& \frac{4y}{p^2\eta}B^{\frac 2p-1}+\frac{2y^2}{p^2\eta} \left(\frac 2p-1 \right) B'_y B^{\frac 2p-2}+2z+2yz'_y, \nonumber \\
M'_\eta &=& -\frac{2y^2}{p^2\eta^2}B^{\frac 2p-1}+\frac{2y^2}{p^2\eta} \left(\frac 2p-1 \right) B'_\eta B^{\frac 2p-2}+2yz'_\eta, \nonumber \\
\end{eqnarray*}
\begin{eqnarray*}
T'_\eta &=& \frac{2y}{p\eta^2}B^{\frac 2p}+\frac{-4y}{p^2\eta}  B'_\eta B^{\frac 2p-1}.
\end{eqnarray*}
By the above definitions of $S$, $M$, and $T$, we have
\begin{eqnarray*}
z''_{yy}=\frac{S'_yM-M'_yS}{M^2}, \ \ \ z''_{\eta y} =\frac{S'_\eta M-M'_\eta S}{M^2}, \ \ \ z''_{\eta \eta}=\frac{T'_\eta M-M'_\eta T}{M^2}.
\end{eqnarray*}
The first and second derivatives of $\Phi$ are calculated as follows:
\begin{eqnarray}
\nabla \Phi &=&  \left [ \begin{array} {c}   \frac{2z}{t^{\frac 2p}-z^2}  \\  \\ \frac{-\frac 2p t^{\frac 2p-1}}{t^{\frac 2p}-z^2}- \frac 2t  \end{array}\right],  \nonumber \\
 \ \ \nabla^2 \Phi  &=&     \left [ \begin{array} {cc}  \frac{2(t^{\frac 2p}-z^2)+4z^2}{(t^{\frac 2p}-z^2)^2}  &  \frac{- \frac 4pt^{\frac 2p-1}z}{(t^{\frac 2p}-z^2)^2} \\
    & \\
 \frac{- \frac 4pt^{\frac 2p-1}z}{(t^{\frac 2p}-z^2)^2}  & \frac{-\frac 2p \left( \frac 2p-1\right)t^{\frac 2p-2}(t^{\frac 2p}-z^2) +\left( \frac 2p\right)^2 t^{\frac 4p-2} }{(t^{\frac 2p}-z^2)^2} +\frac{2}{t^2} \end{array}\right].
\end{eqnarray}
The first and second derivatives of $\Phi_*$ are messier. For the first derivative we have
 \begin{eqnarray}
\nabla \Phi_* &=&  \left [ \begin{array} {c}   -\left( \frac 1p-1\right)(z+yz'_y) + \frac{\frac{2}{p} B'_y B^{\frac 2p-1} - 2zz'_y}{B^{\frac 2p} - z^2} + \frac{2B'_y}{B}  \\  \\ -\left( \frac 1p-1\right)(yz'_\eta) + \frac{\frac{2}{p} B'_\eta B^{\frac 2p-1} - 2zz'_\eta}{B^{\frac 2p} - z^2} + \frac{2B'_\eta}{B} \end{array}\right],  \nonumber \\
\end{eqnarray}
For calculating the second derivative, we also need the second derivatives of $B$:
\begin{eqnarray*}
B''_{yy} &=& \frac{2z'_y+z''_{yy}}{-p\eta}, \nonumber \\
B''_{y\eta} &=&  \frac{-p\eta (z'_\eta+y z''_{\eta y}) + p(z+y z'_y)}{(p\eta)^2} \nonumber \\
B''_{\eta \eta} &=& -\frac{yz''_{\eta \eta}\eta-yz'_\eta}{\eta^2} - \left ( \frac 1p+1\right) \frac{4}{\eta^3} + \frac 1p \frac{yz'_{\eta} \eta-zy}{\eta^2 }. 
\end{eqnarray*}
Using the second derivatives of $B$, we have
\begin{eqnarray}
\nabla^2 \Phi_* =  \left [ \begin{array} {cc} f_{11}  &  f_{12} 
  \\ f_{21}  & f_{22}  \end{array}\right],
\end{eqnarray}
where 
\begin{eqnarray*}
f_{11} &=& -\left( \frac 1p-1\right) (2z'_y+yz''_{yy}) + \frac{\left[ \frac 2p \left[ B''_{yy} B^{\frac 2p-1} + \left( \frac 2p -1 \right) (B'_y)^2  B^{\frac 2p-2}\right] -2((z'_y)^2+zz''_{yy}) \right] (B^{\frac 2p} - z^2)  }{(B^{\frac 2p} - z^2)^2}      \nonumber \\
&& - \frac{\left[ \frac 2p B'_y B^{\frac 2p-1} - 2zz'_y\right]^2 }{(B^{\frac 2p} - z^2)^2} + \frac{2B''_{yy}B-2(B'_y)^2}{B^2}.
\end{eqnarray*}
$f_{21}=f_{12}$ and $f_{22}$ have similar formulations. 

\begin{tabular}{  | c | c| }
    \hline
      { $\Phi(z,t)$}  & $\Phi_*(y,\eta)$ \\  \hline
     $-\ln(z^p+t)-\ln(z)$ &  $\eta (p-1) z^p(y,\eta) - 2 - \ln(-\eta) + \ln(z(y,\eta))$ \\ \hline
  \end{tabular}
  
where $z$ is the solution of 
\begin{eqnarray} 
&& y-\eta p z^{(p-1)}+\frac 1z=0  
\end{eqnarray}
Similar to the previous case, for simplicity, we drop the arguments of $z(y,\eta)$ and denote it as $z$. We denote the first derivatives with respect to $y$ and $\eta$ as $z'_y$ and $z'_\eta$, respectively. 
By implicit differentiation, we have
\begin{eqnarray*}
z'_y &=& \frac{1}{\eta p(p-1)z^{p-2}+z^{-2}   \ \ \ =: B}, \nonumber \\
z'_\eta &=& \frac{-pz^{p-1}}{\eta p(p-1)z^{p-2}+z^{-2}}.
\end{eqnarray*}
For the second derivatives of $z$, by using 
\begin{eqnarray*}
B'_y &=& \eta p(p-1)(p-2)z'_yz^{p-3}-2z'_yz^{-3}, \nonumber \\
B'_\eta &=& p(p-1)z^{p-2} + \eta p(p-1)(p-2)z'_\eta z^{p-3}-2z'_\eta z^{-3}.
\end{eqnarray*}
we have
\begin{eqnarray*}
z''_{yy}=\frac{-B'_y}{B^2}, \ \ \ z''_{\eta y} =\frac{-B'_\eta}{B^2}, \ \ \ z''_{\eta \eta}=\frac{-p(p-1)z'_\eta z^{p-2} B+pz^{p-1}B'_\eta}{B^2}.
\end{eqnarray*}
The first and second derivatives of $\Phi$ are calculated as follows:
\begin{eqnarray*}
\nabla \Phi &=&  \left [ \begin{array} {c}   -\frac{pz^{p-1}}{z^p+t} -\frac 1x  \\  -\frac{1}{z^p+t}  \end{array}\right],  \nonumber \\
 \ \ \nabla^2 \Phi  &=&     \left [ \begin{array} {cc}  -p(p-1)z^{p-2}(z^p+t)+p^2z^{2p-2} +\frac{1}{z^2} &  \frac{pz^{p-1}}{(z^p+t)^2} \\
 \frac{pz^{p-1}}{(z^p+t)^2}  & \frac{1}{(z^p+t)^2}\end{array}\right].
\end{eqnarray*}
The first derivative of $\Phi_*$ is equal to
 \begin{eqnarray}
\nabla \Phi_* &=&  \left [ \begin{array} {c}   \eta p(p-1)z'_y x^{p-1} +\frac{z'_y}{z} \\  (p-1) z^p + \eta p(p-1)z'_\eta z^{p-1} - \frac 1\eta + \frac{z'_\eta}{z}\end{array}\right],  \nonumber \\
\end{eqnarray}
and the second derivatives of $\Phi_*$ is equal to
\begin{eqnarray}
\nabla^2 \Phi_* =  \left [ \begin{array} {cc} f_{11}  &  f_{12} 
  \\ f_{21}  & f_{22}  \end{array}\right],
\end{eqnarray}
where 
\begin{eqnarray*}
f_{11} &=& \eta(p-1)p \left [ z'_{yy}x^{p-1} +(p-1)(z'_y)^2 z^{p-2}\right] + \frac{z'_{yy}z-(z'_y)^2}{z^2}, \nonumber \\
f_{12}=f_{21} &=& (p-1)p \left[ z'_yz^{p-1} +\eta z'_{y\eta}z^{p-1} + \eta(p-1)z'_y z'_\eta z^{p-2}\right]+\frac{z'_{y\eta}z-z'_\eta z'_y}{z^2}, \nonumber \\
f_{22} &=& (p-1)pz'_\eta z^{p-1} +(p-1)p \left[ z'_\eta z^{p-1} + \eta z'_{\eta \eta} z^{p-1} + \eta (z'_\eta)^2 (p-1)z^{p-2}\right] +\frac{1}{\eta^2} + \frac{z'_{\eta \eta}z-(z'_\eta)^2}{z^2}.
\end{eqnarray*}

 \begin{tabular}{  | c | c| }
    \hline
      { $\Phi(z,t)$}  & $\Phi_*(y,\eta)$ \\  \hline
     $-\ln(zt-1)$  & $-1-\sqrt{1+4y \eta}+\ln\left(\frac{1+\sqrt{1+4y \eta}}{2y \eta}\right)$ \\ \hline
  \end{tabular}
  
   For the primal function we have
\begin{eqnarray*}  
\nabla \Phi =  \left [ \begin{array} {c}  -\frac{t}{zt-1} \\  -\frac{z}{zt-1} \end{array}\right],
 \ \ \nabla^2 \Phi =  \left [ \begin{array} {cc} \frac{t^2}{(zt-1)^2} &  \frac{1}{(zt-1)^2} \\ \frac{1}{(zt-1)^2}   & \frac{z^2}{(zt-1)^2}   \end{array}\right],
\end{eqnarray*}
and for the dual function we have
\begin{eqnarray*} 
\nabla \Phi_* =  \left [ \begin{array} {c}  -\frac{2\eta}{1+\sqrt{1+4y\eta}}-\frac{1}{y}  \\ -\frac{2y}{1+\sqrt{1+4y\eta}}-\frac{1}{\eta}  \end{array}\right],
 \ \ \nabla^2 \Phi_* =  \left [ \begin{array} {cc} \frac{4\eta^2}{\sqrt{1+4y\eta}(1+\sqrt{1+4y\eta})^2}+\frac{1}{y^2} &   \frac{-2(\sqrt{1+4y\eta}+1+4y\eta)+4y\eta}{\sqrt{1+4y\eta}(1+\sqrt{1+4y\eta})^2} 
  \\\frac{-2(\sqrt{1+4y\eta}+1+4y\eta)+4y\eta}{\sqrt{1+4y\eta}(1+\sqrt{1+4y\eta})^2}   & \frac{4y^2}{\sqrt{1+4y\eta}(1+\sqrt{1+4y\eta})^2}+\frac{1}{\eta^2}   \end{array}\right].
\end{eqnarray*}

\section{A s.c.\ barrier for vector relative entropy}  \label{sec:vRE}

In section, we prove that the function
\begin{eqnarray} \label{eq:RE-1-2}
\Phi(t,u,z) := -\ln \left(t-\sum_{i=1}^{\ell} u_i\ln(u_i / z_i)\right)- \sum_{i=1}^{\ell}\ln(z_i)-\sum_{i=1}^{\ell}\ln(u_i),
\end{eqnarray}
is a $(2\ell+1)$-LH-s.c.\ barrier for the epigraph of vector relative entropy function $f: \R_{++}^\ell \oplus \R_{++}^\ell \rightarrow \R$ defined as 
\[
f(z,u):= \sum_{i=1}^{\ell} u_i\ln(u_i) - u_i\ln(z_i).
\] 
First note that $\Phi$ is $(2\ell+1)$-LH, so we just need to prove self-concordance.

The proof is based on the compatibility results of \cite{interior-book}. We first consider the case $\ell = 1$ and then generalize it. Let us define the map $A: \R_{++}^2 \rightarrow \R$ as
\begin{eqnarray}\label{eq:rel-13}
A(u,z)=-u\ln \left( \frac uz\right). 
\end{eqnarray}
For a vector $d:=(d_u,d_z)$, we can verify
\begin{eqnarray}\label{eq:rel-14}
A''[d,d]&=& -\frac{(zd_u-ud_z)^2}{uz^2}, \nonumber \\
A'''[d,d,d]&=&\frac{(zd_u+2ud_z)(zd_u-ud_z)^2}{u^2z^3}.
\end{eqnarray}
\eqref{eq:rel-14} implies that $A$ is concave with respect to $\R_+$. We claim that $A$ is $(\R_+,1)$-compatible with the barrier $-\ln(u)-\ln(z)$ (\cite{interior-book}- Definition 5.1.2). For this, we need to show 
\begin{eqnarray}\label{eq:rel-15}
A'''[d,d,d]\leq -3 A''[d,d] \sqrt{\frac{z^2d_u^2+u^2d_z^2}{u^2z^2}}. 
\end{eqnarray}
By using \eqref{eq:rel-14} and canceling out the common terms from both sides, \eqref{eq:rel-15} reduces to 
\begin{eqnarray}\label{eq:rel-16}
 (zd_u+2ud_z) \leq 3 \sqrt{z^2d_u^2+u^2d_z^2}. 
\end{eqnarray}
We can assume that the LHS is nonnegative, then by taking the square of both side and reordering, we get the obvious inequality 
\begin{eqnarray}\label{eq:rel-17}
8z^2d_u^2-4zd_yud_z+5u^2d_z^2=4z^2d_u^2+(2zd_u-ud_z)^2+4u^2d_z^2 \geq 0. 
\end{eqnarray}
Therefore $A(u,z)$ is $(\R_+,1)$-compatible with the barrier $-\ln(u)-\ln(z)$. Also note that its summation with a linear term $t+A(u,z)$ is also $(\R_+,1)$-compatible with the barrier $-\ln(z)-\ln(u)$.  Hence,  $-\ln(t+A)-\ln(u)-\ln(z)$ is a 3-s.c.\ barrier by \cite{interior-book}-Proposition 5.1.7. 

For the general case, consider the map $\bar A: [\R_{++}^2]^\ell \rightarrow \R$ as
\begin{eqnarray}\label{eq:rel-13-2}
\bar A(u,z) := \sum _{i}^\ell A(u_i,z_i).
\end{eqnarray}
For a vector $d$ of proper size, we have
\begin{eqnarray}\label{eq:rel-14-2}
\bar A''(u,z)[d,d]&=&  \sum_{i}^\ell A''(u_i,z_i)[d^i,d^i], \nonumber \\
\bar A'''(u,z)[d,d,d]&=&\sum_{i}^\ell A'''(u_i,z_i)[d^i,d^i,d^i].
\end{eqnarray}
We claim that $\bar A$ is $(\R_+,1)$-compatible with the barrier $-\sum_i \ln(u_i)- \sum_i \ln(z_i)$.
First note that $\bar A$ is concave with respect to $\R_+$. We need to prove a similar inequality as \eqref{eq:rel-15} for $\bar A$. Clearly we have
\begin{eqnarray} \label{eq:rel-14-3}
\frac{z_j^2(d^j_u)^2+u_j^2(d^j_z)^2}{u_j^2z_j^2}  \leq \sum_i^\ell \frac{z_i^2(d^i_u)^2+u_i^2(d^i_z)^2}{u_i^2z_i^2}, \ \ \forall j \in \{1,\ldots,\ell\}.
\end{eqnarray}
Using inequality \eqref{eq:rel-15} and \eqref{eq:rel-14-3} for all $j$ and adding them together yields the inequality we want for $\bar A$. Therefore, $\bar A$ is $(\R_+,1)$-compatible with the barrier $-\sum_i \ln(u_i)- \sum_i \ln(z_i)$, and by \cite{interior-book}-Proposition 5.1.7,  \eqref{eq:RE-1-2} is a $(2\ell+1)$-s.c.\ barrier.

\section{Comparison with some related solvers and modeling systems} \label{sec:other-solvers}
In this section, we take a look at the input format for some other well-known solvers and modeling systems. \cite{mittelmann2012state} is a survey by Mittelmann about solvers for conic optimization, which gives an overview of the major codes available for the solution of linear semidefinite (SDP) and second-order cone (SOCP) programs. Many of these codes also solve linear programs (LP). We mention the leaders MOSEK, SDPT3, and SeDuMi from the list. We also look at CVX, a very user-friendly interface for convex optimization. CVX is not a solver, but is a modeling system that (by following some rules) detects if a given problem is convex and remodels it as a suitable input for solvers such as SeDuMi. 

\subsection{MOSEK \cite{mosek}} MOSEK is a leading commercial solver for not just optimization over symmetric cones, but also many other convex optimization problems. The most recent version, MOSEK 9.0, for this state-of-the-art convex optimization software handles, in a primal-dual framework, many convex cone constraints 	which arise in applications \cite{dahl2019primal}.  There are different options for the using platform that can seen in MOSEK's website \cite{mosek}.

\subsection{SDPT3 \cite{SDPT3-2003, SDPT3-user-guide}}
SDPT3 is a MATLAB package for optimization over symmetric cones, and it solves a conic optimization problem in the equality form as 
\begin{eqnarray} \label{conic problems-sedumi}
\begin{array} {cc}  \min & \langle c, x \rangle \\
                                 \text{s.t.} &  A x=b, \\
                                   & x \in K,
    \end{array}
\end{eqnarray}
where our cone $K$ can be a direct sum of  nonnegative rays (leading to LP problems), second-order cones or semidefinite cones. 
The Input for SDPT3 is given in the cell array structure of MATLAB.  The command to solve SDPT3 is of he form
\begin{verbatim}
[obj,X,y,Z,info,runhist] = sqlp(blk,At,C,b,OPTIONS,X0,y0,Z0).
\end{verbatim}
The input data is given in different blocks, where for the $k$th block, \tx{blk\{k,1\}} specifies the type of the constraint. Letters \tx{'l'}, \tx{'q'}, and \tx{'s'} are representing linear, quadratic, and semidefinite constraints.  In the $k$th block, \tx{At\{k\}}, \tx{C\{k\}}, ... contain the part of the input related to this block. 

\subsection{SeDuMi \cite{sedumi}}
SeDuMi is also a MATLAB package for optimization over symmetric cones in the format of \eqref{conic problems-sedumi}.
For SeDuMi, we give as the input $A$, $b$ and $c$ and a structure array $K$.  The vector of variables has a ``direct sum" structure. In other words, the set of variables is the direct sum of  free, linear, quadratic, or semidefinite variables. The fields of the structure array $K$ contain the number of constraints we have from each type and their sizes. SeDuMi can be called in MATLAB by the command
\begin{verbatim}
[x,y] = sedumi(A,b,c,K);
\end{verbatim}
and the variables are distinguished by $K$ as follows:
\begin{enumerate}
\item{$K.f$ is the number of free variables, i.e., in the variable vector $x$, \texttt{x(1:K.f)} are free variables.  }
\item{$K.l$ is the number of nonnegative variables.  }
\item{$K.q$ lists the dimension of Lorentz constraints.  }  	
\item{$K.s$ lists the dimensions of positive semidefinite constraints.  }
\end{enumerate}
For example, if \texttt{K.l=10}, \texttt{K.q=[3 7]} and \text{K.s=[4 3]}, then \texttt{x(1:10)} are non-negative. Then we have \texttt{x(11) >= norm(x(12:13))}, \texttt{x(14) >= norm(x(15:20))}, and \texttt{mat(x(21:36),4)} and \texttt{mat(x(37:45),3)} are positive semidefinite. 
 To insert our problem into SeDuMi, we have to write it in the format of \eqref{conic problems-sedumi} .  We also have the choice to solve the dual problem because all of the above cones are self-dual. 
 
\subsection{CVX \cite{cvx}}
CVX is an interface that is more user-friendly than solvers like SeDuMi. It provides many options for giving the problem as an input, and then translates them to an eligible format for a solver such as SeDuMi. We can insert our problem constraint-by-constraint into CVX, but they must follow a protocol called \emph{Disciplined convex programming} (DCP).  DCP has a rule-set that the user has to follow, which allows CVX to verify that the problem is convex and convert it to a solvable form. For example, we can write a $<=$ constraint only when the left side is convex and the right side is concave, and to do that, we can use a large class of functions from the library of CVX.


\end{document}